\documentclass[11 pt]{amsart}

\usepackage{times}
\usepackage{geometry}
\usepackage{amssymb}
\usepackage{latexsym,amssymb,  amsmath, amscd, amsfonts}
\usepackage{graphicx}
\usepackage[percent]{overpic}
\usepackage{pdfsync}
\usepackage{units}
\usepackage{hyperref}
\usepackage{euscript}
\usepackage{multicol}
\usepackage{epstopdf}
\usepackage{paralist}

\newtheorem{theorem}{Theorem}
\newtheorem{lemma}[theorem]{Lemma}
\newtheorem{proposition}[theorem]{Proposition}
\newtheorem{corollary}[theorem]{Corollary}

\theoremstyle{definition}
\newtheorem{definition}[theorem]{Definition}

\newtheorem{conjecture}[theorem]{Conjecture}

\newtheorem{remark}[theorem]{Remark}

\newcommand{\R}{\mathbb{R}}      
\newcommand{\Z}{\mathbb{Z}}
\newcommand{\N}{\mathbb{N}}
\def\Lk{{\operatorname{Lk}}}
\def\Tw{{\operatorname{Tw}}}
\def\Wr{{\operatorname{Wr}}}
\def\Rib{{\operatorname{Rib}}}
\def\Cr{{\operatorname{Cr}}}
\def\Len{{\operatorname{Len}}}
\newcommand{\mk}{\mathcal{K}}
\newcommand{\um}{\mathcal{U}}
\newcommand{\kwf}{\mathcal{K}_{w,F}}
\newcommand{\kw}{\mathcal{K}_{w}}
\newcommand{\kpwf}{\mathcal{K}'_{w,F}}
\newcommand{\uwf}{\mathcal{U}_{w,F}}
\newcommand{\uw}{\mathcal{U}_{w}}
\newcommand{\Ai}{\frac{\alpha_1}{2}}
\newcommand{\Aii}{\frac{\alpha_2}{2}}
\newcommand{\Aiii}{\frac{\alpha_3}{2}}
\newcommand{\w}{\frac{w}{2}}
\newcommand{\rin}{r_{in}}
\begin{document}

\title{Linking number and folded ribbon unknots}
\author[Denne]{Elizabeth Denne}
\address{Elizabeth Denne: Washington \& Lee University, Department of Mathematics, Lexington VA}
\email[Corresponding author]{dennee@wlu.edu}
\urladdr{https://elizabethdenne.academic.wlu.edu/}
\author[Larsen]{Troy Larsen}
\address{Troy Larsen: Washington \& Lee University}
\date{\today}
\makeatletter								
\@namedef{subjclassname@2020}{%
  \textup{2020} Mathematics Subject Classification}
\makeatother


\begin{abstract}
We study Kauffman's model of folded ribbon knots: knots made of a thin strip of paper folded flat in the plane. The folded ribbonlength is the length to width ratio of such a folded ribbon knot. The folded ribbon knot is also a framed knot, and the ribbon linking number is the linking number of the knot and one boundary component of the ribbon. We find the minimum folded ribbonlength for $3$-stick unknots with ribbon linking numbers $\pm1$ and $\pm 3$, and we prove that the minimum folded ribbonlength for $n$-gons with obtuse interior angles is achieved when the $n$-gon is regular. Among other results, we prove that the minimum folded ribbonlength of any folded ribbon unknot which is a topological annulus with ribbon linking number $\pm n$ is bounded from above by~$2n$.
 \end{abstract}

\maketitle

\section{Introduction}\label{sect:introduction}

Take a long thin strip of paper, tie a trefoil knot in it, then gently tighten and flatten it. As can be seen in Figure~\ref{fig:trefoil-pentagon}, the boundary of the knot forms a pentagon. This observation is well-known in recreational mathematics \cite{Ashley, John, Wel}. L. Kauffman \cite{Kauf05}  introduced a mathematical model of such a {\em folded ribbon knot}. Kauffman viewed the ribbon as a set of rays parallel to a polygonal knot diagram with the folds acting as mirrors and the over-under information appropriately preserved. An overview of the history of folded ribbon knots can be found in \cite{Den-FRS}. 
\begin{center}
\begin{figure}[htbp]
\begin{overpic}{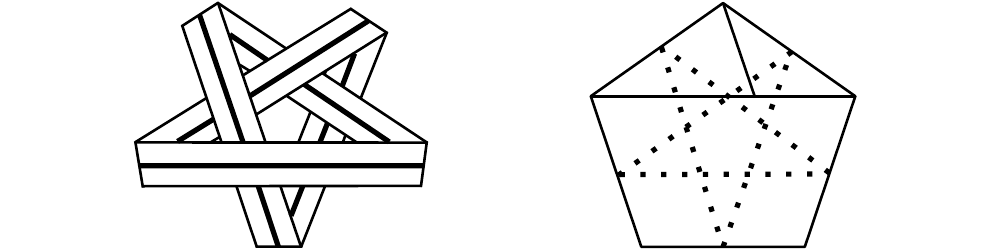}
\end{overpic}
\caption{On the left a folded ribbon trefoil knot. On the right, the folded trefoil knot has been ``pulled tight'', minimizing the folded ribbonlength. Figure used with permission from \cite{Den-FRC}.}
\label{fig:trefoil-pentagon}
\end{figure}
\end{center}

One of the beautiful things about folded ribbon knots is that, while they have a formal mathematical description (given in Section~\ref{sect:model}), they can also be constructed from strips of paper. We strongly encourage the reader to have strips of paper handy as they read along. Many of the constructions given in later sections will make the most sense when made with a physical model.

For a knot (or link) $\mk$, we denote the corresponding folded ribbon knot as $\kwf$, where $w$ is the width of the ribbon and $F$ is the choice of folds (as described in Section~\ref{sect:model}). We seek to find the {\em folded ribbonlength} $\Rib(\kwf$), the ratio of the length of the knot to the width of the ribbon around it. One of the main open problems in the theory of folded ribbons is to find the minimum folded ribbonlength for any knot or link type $K=[\mk]$. A typical way to attack this problem is to find a particular example of a folded ribbon knot, then calculate the folded ribbonlength. This gives an upper bound of the minimum folded ribbonlength for that knot type.  The majority of the papers written to date have taken that strategy. Here is the summary of what is known about the upper bounds on the folded ribbonlength of various knot types. We have given references followed by the best upper bounds to date. 
\begin{enumerate}
\item Trefoil knot \cite{Den-FRF, Kauf05, KMRT}, with $\Rib(K)\leq 6$ from \cite{Den-FRF}.
\item $(p,q)$-torus link \cite{DeM, Den-FRF, KMRT, Tian}, with $\Rib(K)\leq 2p$ where $p\geq q\geq 2$ from \cite{Den-FRF}.
\item Figure 8 knot \cite{Den-FRF, Kauf05, Tian}, with $\Rib(K)\leq 10$ from \cite{Den-FRF}
\item Twist knot with $n$ twists \cite{Den-FRF, Tian}, with 
$$\Rib(K)\leq \begin{cases} 2n+6  \quad \text{ for $n\leq 8$ from \cite{Den-FRF}}, \\
 \frac{\sqrt{5}+1}{2}n+5+\sqrt{5}+\sqrt{\frac{5+\sqrt{5}}{2}}  \quad\text{ for $n\geq 9$ from \cite{Tian}}.
 \end{cases}$$
\item $2$-bridge knot $K$ \cite{Den-FRF}, with crossing number $\Cr(K)$ has $\Rib(K)\leq  6\Cr(K)-2$.
\item $(p,q,r)$-pretzel knot \cite{Den-FRF}, with $\Rib(K)\leq 2(|p|+|q|+|r|) + 2$.
\end{enumerate}
There is work in-progress \cite{KNY} showing that a $2$-bridge link $K$ with crossing number $\Cr(K)$ has $\Rib(K)\leq 2\Cr(K)+2$. Also, any twisted torus knot $K=T_{p,q;r,s}$ (see \cite{Dean,Lee}) has
$$\Rib(K) \leq \begin{cases} 2(\max\{p,q,r\}+|s|r \quad \text{for $r<p+q$ and $s\neq 0$,}
\\ 2(p+(|s|-1)r) \quad \text{for $r\leq p-q$.}
\end{cases}
$$

Finding lower bounds on folded ribbonlength appears to be harder. In \cite{DKTZ}, E. Denne {\em et al.} make a first pass at this problem for $3$-stick unknots. This work is reviewed at the start of Section~\ref{sect:3-stick}. We make some progress on finding lower bounds in this paper, as described further below.

A second open problem in the theory of folded ribbon knots is to relate the minimum folded ribbonlength of a knot type $K$ to its crossing number $\Cr(K)$. The {\em ribbonlength crossing number problem} asks to find positive constants $c_1, c_2, \alpha, \beta$ such that
\begin{equation*}
c_1\cdot \Cr(K)^\alpha \leq \Rib(K) \leq c_2\cdot \Cr(K)^\beta.
\label{eq:crossing}
\end{equation*}
Y. Diao and R. Kusner \cite{DK} conjecture that $\alpha=\frac{1}{2}$ and $\beta=1$.
Other people \cite{Den-FRC, Den-FRF, KMRT, Tian} have made progress on this problem, including proving $\beta=1$ for many infinite knot families, and $\beta=\frac{3}{2}$ for any knot type. This open problem is not the focus of this paper.  

In this paper, we recognize that folded ribbon knots are framed knots. We seek to minimize folded ribbonlength while respecting the knot type, the framing of the folded ribbon knot, and the topological type of the folded ribbon knot. That is, with respect to {\em ribbon equivalence}. This is a challenging problem, and in this paper we make a start by working with unknots.

In Section~\ref{sect:defn}, we remind the reader of the formal definition of a folded ribbon knot and folded ribbonlength. We also review the local geometry of a fold in a folded ribbon knot. The ribbonlength of the pieces of ribbon in a fold found in Proposition~\ref{prop:extended-fold} and Corollary~\ref{cor:triangles} are used extensively.

We formally define {\em ribbon equivalence} in Section~\ref{sect:ribbon-equivalence}. In particular, we keep track of the framing of the folded ribbon knot by computing the {\em ribbon linking number}. This is the linking number between the knot and one boundary component of the ribbon. There are a few shortcuts that we have developed for computing ribbon linking number, and these can be found in Section~\ref{sect:linking-number}. Of note is Lemma~\ref{lem:fold-compute}, where we give the sign that any given fold contributes to the ribbon linking number.  We use this insight together with the minimum folded ribbonlength needed to create a fold (from Corollary~\ref{cor:triangles}) to give our first lower bound on the minimum folded ribbonlength in Theorem~\ref{thm:RibLowerBd}:  suppose the folded ribbon knot $\kwf$ has writhe $\Wr(\kwf) =0$ and ribbon linking number $\Lk(\kwf)=\pm n$. Then the minimum folded ribbonlength of such a folded ribbon knot is bounded from below as follows:
\begin{compactenum} 
\item $2n \leq \Rib([\kwf])$  \ \  when $\kwf$ is a topological annulus,
\item $n \leq \Rib([\kwf])$ \ \ when $\kwf$ is a M\"obius band. 
\end{compactenum} 
We finish Section~\ref{sect:linking-number} by computing the possible ribbon linking numbers of folded ribbon unknots whose  knot diagrams are convex $n$-gons. In Proposition~\ref{prop:n-linking}, we prove that the ribbon linking number is a complete invariant for these unknots. In other words, the folding information determines the ribbon linking number, and the ribbon linking number determines the folding information (up to permutation among the vertices). 

In Section~\ref{sect:3-stick}, we work exclusively with $3$-stick unknots. We build on older work from \cite{DKTZ}  and prove that the minimum folded ribbonlength occurs when the unknot is an equilateral triangle. In this case the folded ribbon unknot is a M\"obius band and either has ribbon linking number $\pm 1$ (Corollary~\ref{cor:3-stick-1}) or ribbon linking number $\pm 3$ (Theorem~\ref{thm:3-stick-3}). In Section~\ref{sect:n-stick}, we  generalize these arguments to folded ribbon $n$-stick unknots. The nature of the geometry of the folds lead us to restrict our attention to $n$-stick unknots whose fold angles $\alpha_i$ satisfy $\frac{\pi}{2}\leq \alpha_i <\pi$. In Corollary~\ref{cor:minngon}, we prove that the minimum folded ribbonlength occurs when the unknot is a regular $n$-gon. Knowing the folded ribbonlength of a regular $n$-gon and the possible ribbon linking numbers (Proposition~\ref{prop:n-linking}) allows us, in Theorem~\ref{thm:n-unknot1},
 to find upper bounds on the minimum folded ribbonlength for folded ribbon unknots which are M\"obius bands and have ribbon linking number $\pm n$.
 
In Section~\ref{sect:unknots}, we only consider folded ribbon unknots which are topological annuli. We broaden our view and look at unknots whose knot diagrams are not convex and not regular. Our main result is Theorem~\ref{thm:linking-n}, which states that there is a folded ribbon unknot $\uwf$ which is a topological annulus,  such that
\begin{compactenum}
\item for any $n\in \N$, the knot diagram $\um$ has $(2n+2)$ sticks,
\item $\uwf$ has ribbon linking number $\Lk(\uwf)=\pm n$, and 
\item $\uwf$ has folded ribbonlength $\Rib(\uwf) = 2n$.
\end{compactenum}
In our final result, we consider folded ribbon knots where the knot $\mk$ is non-trivial. We create a connected sum between the knot $\mk$ and and unknot $\um$ from Theorem~\ref{thm:linking-n} which has any desired ribbon linking number. This allows us to find an upper bound on the minimum folded ribbonlength for a non-trivial knot whose corresponding folded ribbon has any ribbon linking number.

\section{Definitions and Background}\label{sect:defn}

\subsection{Modeling folded ribbon knots}\label{sect:model}

Following knot theory texts \cite{Adams, Crom, Liv}, we define a {\em knot} $\mk$ to be a simple closed curve in $\R^3$ that is equivalent to a polygonal knot.  Two knots are equivalent if they are ambient isotopic, and the equivalence class of a knot $\mk$ is denoted by $K=[\mk]$. A {\em link} is a finite disjoint union of knots. The definitions in this paper also hold for links.  A {\em polygonal knot} $\mk$ has a finite number of vertices $v_1,\dots, v_n$ and edges $e_1=[v_1,v_2], \dots, e_n=[v_n,v_1]$.  If $\mk$ is oriented, we assume that the numbering of the vertices follows the orientation. Sometimes the edges of a polygonal knot are called {\em sticks}. 
A {\em polygonal knot diagram} is a projection of a polygonal knot to a plane, where the crossing information of overlapping strands has been preserved. This is usually depicted at a crossing where the ``lower strand" has a break at the intersection.

Formally, a {\em folded ribbon knot of width $w$}, denoted $\kw$, is a piecewise linear immersion of an annulus or M\"obius band into the plane, where the fold lines are the only singularities, and where the crossing information is consistent. A detailed description can be found in  \cite{Den-FRS, Den-FRF, DKTZ}.  We observe that a folded ribbon knot has two kinds of double points in the plane. The first are those near crossings of the knot diagram,  and the second are those near the fold lines of the ribbon. Specifically, we define a {\em fold} to be a connected component of the set of double points which lifts to a single component containing the fold line in the preimage. For example,  on the left in Figure~\ref{fig:folding-info}, the fold is the set of double points found in triangle $\Delta ABC$.  We also observe that at each fold, there is a choice of which piece of the ribbon lies over the other. Given an orientation of $\kw$, we define an {\em underfold} and {\em overfold} as shown in Figure~\ref{fig:folding-info}. Altogether, these choices give the {\em folding information} $F$ of a folded ribbon knot. 

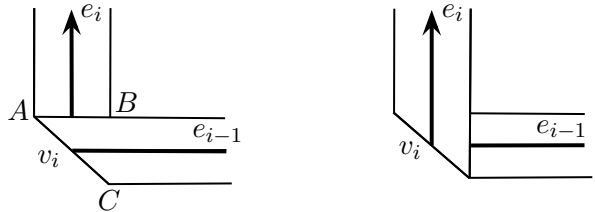
\begin{figure}[htbp]
\begin{center}
\begin{overpic}{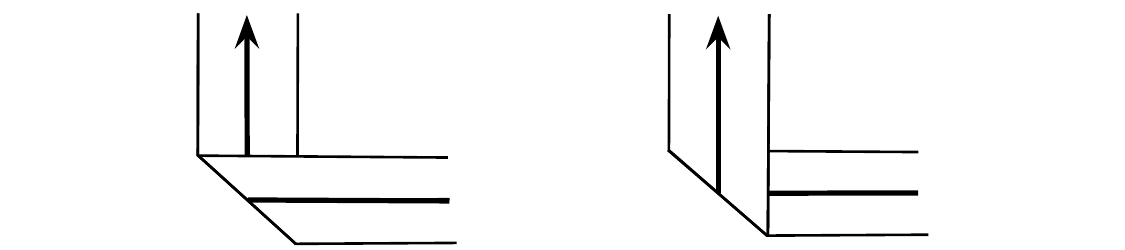}
\put (14.5, 8){$A$}
\put(27, 9){$B$}
\put(25, -2.5){$C$}
\put(36, 6){$e_{i-1}$}
\put(23, 20){$e_{i}$}
\put(18, 3){$v_i$}

\put(76, 6.5){$e_{i-1}$}
\put(65, 20){$e_{i}$}
\put(60, 4){$v_i$}
\end{overpic}
\caption{A right underfold (left) and a right overfold (right). Figure re-used with permission from~\cite{DKTZ}.}
\label{fig:folding-info}
\end{center}
\end{figure}

There is an alternate view of a folded ribbon knot $\kwf$ which accounts for the folding information explicitly. We start by by taking a particular polygonal knot diagram $\mk$ and adding a ribbon of width $w$ about $\mk$ such that the folds at each vertex have folding information~$F$. Here, the boundary of the ribbon is equidistant from and parallel to $\mk$, and each fold line is perpendicular to the angle bisector of the angle at the vertex. 
In this case, we will refer to $\kwf$ as the folded ribbon knot {\em corresponding} to $\mk$. Following  Definition~2.7 in \cite{DKTZ}, we assume that the folded ribbon knot corresponding to $\mk$ has no singularities (is immersed), except at the fold lines which are assumed to be disjoint. We also assume that $\kwf$ has consistent crossing information, and this agrees with the folding information given by $F$, and
 with the crossing information from $\mk$.  We proved in \cite{Den-FRF} that all polygonal knot diagrams have such a folded ribbon knot for small enough widths.
 
As a simple example of these ideas, take a strip of paper and join the ends to create an annulus. Now flatten to create two folds: this is a folded ribbon unknot. This folded ribbon unknot corresponds to the knot diagram with two edges or sticks, where we remember that one stick is over the other.  This example fits our narrative since immersions are locally one-to-one. It also serves to remind us that the polygonal knot diagrams of folded ribbon knots need not be regular.\footnote{A  polygonal knot diagram is {\em regular} (see \cite{Liv}) provided that no three points on the knot project to the same point and no vertex projects to the same point as any other point on the knot.} 

\subsection{Diagram Equivalence and Ribbonlength}\label{sect:ribbonlength}

  Given our understanding of a folded ribbon knot, it is natural to wonder when two folded ribbon knots are equivalent.  The most obvious way to do so is to ignore the folded ribbon altogether and only consider the knot type.
  
\begin{definition}[\cite{DKTZ} Definition 3.4] Two folded ribbon knots are {\em diagram equivalent} if they have equivalent knot diagrams.
\end{definition}

There are other ways to define ribbon equivalence, and we discuss these in Section~\ref{sect:ribbon-equivalence}.
A natural question is to try to find the least length of paper needed to tie a folded ribbon knot. We therefore define a scale-invariant quantity called {\em folded ribbonlength}.

\begin{definition}[\cite{Kauf05}] The {\em folded ribbonlength} of a particular folded ribbon knot $\kwf$  of width $w$  and folding information $F$ is the quotient of the length of $\mk$ to the width $w$:
$$ \Rib(\mathcal{K}_{w,F}) := \frac{\Len(\mathcal{K})}{w}.$$
\end{definition}

The {\em (folded) ribbonlength problem} asks us to ``minimize'' the folded ribbonlength of a folded ribbon knot in its knot type; that is, with respect to diagram equivalence. For example,  consider the folded ribbon unknot with a knot diagram $\um$ consisting of just two edges.  After setting the width $w=1$, we see that $\Len(\um)\rightarrow 0$. We say the (minimum) ribbonlength of the unknot is $0$. More formally, we define:

\begin{definition}
The {\em folded ribbonlength} of a knot (or link) type $K$ is defined to be
$$\Rib(K)= \Rib([\mk]) = \inf_{\mk\in[\mk]} \Rib(\kwf).$$
\end{definition}

We can easily find upper bounds for the folded ribbonlength of a knot type $K=[\mk]$. Specifically, we can set width $w=1$ and find the length of a knot diagram $\mk$ of a folded ribbon knot $\kw$. This computation then gives the folded ribbonlength for $\kw$, and an upper bound on the folded ribbonlength for $[\mk]$. The majority of the papers written to date fall into this context, see for instance \cite{DeM,  Den-FRS, Den-FRC,
 Kauf05, KMRT, Tian}.  In this paper, we aim to minimize ribbonlength with respect to ribbon equivalence (defined in Section~\ref{sect:ribbon-equivalence}).

\subsection{Local geometry of folded ribbons}\label{sect:useful-geom}

To help with our later computations of folded ribbonlength, we now provide a more detailed description of the geometry of a fold.  In Figure~\ref{fig:AcuteVsObtuse}, we see two folds (the double points in $\Delta ABF$). Their exact geometry depends on whether the fold angle $\theta$ is acute or obtuse.  We have oriented the knot diagram and ribbon in Figure~\ref{fig:AcuteVsObtuse} so that in both cases the fold $\Delta ABF$ is a left underfold. The {\em fold angle} $\theta:=\angle HCD$ at vertex $C$ is shown.\footnote{Following Definition 2.4 in  \cite{DKTZ}, we assume that the fold angle is the angle $\theta$ at a vertex $v_i$ between adjacent edges $e_{i-1}$ and $e_i$, where $0\leq \theta \leq \pi$. The angle is positive when $e_i$ is to the left of $e_{i-1}$ and is negative when $e_i$ is to the right.} Since we use the knot diagram to construct our folded ribbon knot, we can talk about the pieces of the ribbon that correspond to particular parts of the knot diagram.

\begin{center}
\begin{figure}[htbp]
\begin{overpic}{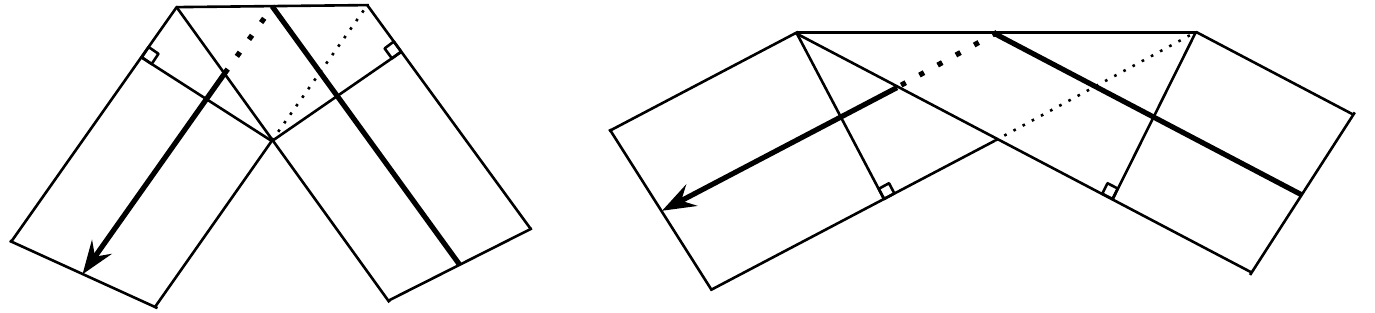}
\put(12,23){$A$}
\put(19,23){$C$}
\put(19.15, 19){$\theta$}
\put(26,23){$B$}
\put(17,16.5){$D$}
\put(11.5,14){$E$}
\put(18.5,9){$F$}
\put(8,18.5){$G$}
\put(20,16.5){$H$}
\put(23.5, 13){$I$}
\put(29.5, 19){$J$}
\put(57.5,21){$A$}
\put(72,21){$C$}
\put(71.5, 17.5){$\theta$}
\put(86,21){$B$}
\put(67.5,15.5){$D$}
\put(57.5,14){$E$}
\put(71.5,9.5){$F$}
\put(64,6.25){$G$}
\put(78,13.5){$H$}
\put(83.4,11){$I$}
\put(79, 6){$J$}
\end{overpic}
\caption{The fold from angle $\theta$ consists of two overlapping triangles of ribbon, which correspond to the part of the knot diagram $DC+CH$. Similarly, the  extended fold from angle $\theta$ is the two pieces of ribbon determined by $EC+CI$. Figure re-used with permission from \cite{Den-FRC}.}
\label{fig:AcuteVsObtuse}
\end{figure}
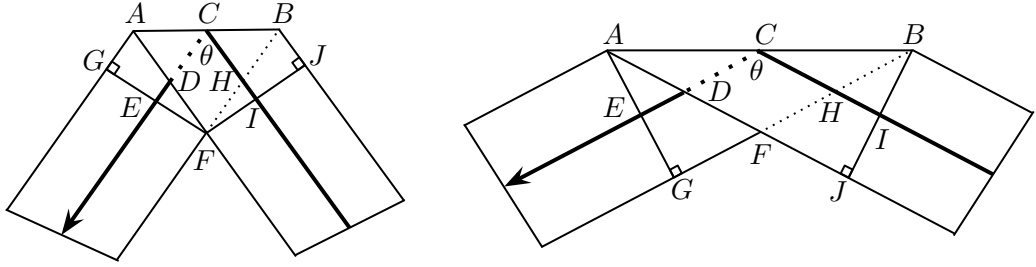
\end{center}

\begin{definition} [\cite{Den-FRF} Definition 9] Assume that the fold angle $\theta$ satisfies $0<\theta < \pi$. Using the notation in Figure~\ref{fig:AcuteVsObtuse}, we define the {\em fold from angle $\theta$} to be the two identical overlapping triangles of ribbon, determined by the knot diagram $DC+CH$. The {\em extended fold from angle $\theta$} is defined to to be the two congruent pieces of ribbon determined by  $EC+CI$. 
\label{def:extended-fold}
\end{definition}

\begin{remark} \label{rmk:fold-angle} Why the restriction on the fold angle $\theta$? Firstly, there is no fold when $\theta=\pi$. The other extreme case occurs when $\theta=0$. Here,  the fold consists of two overlapping rectangles or trapezoids of ribbon.  We saw this kind of fold previously in the $2$-stick unknot case. 
We also note the special case when the fold angle $\theta =\frac{\pi}{2}$.  This angle is the transition point between the acute and obtuse folds shown in Figure~\ref{fig:AcuteVsObtuse}.   In this case, the fold and the extended fold are identical. 
\label{rmk:extended-fold}
\end{remark}

We can use basic trigonometry and the geometry of a folded ribbon (explained fully in \cite{Den-FRF}) to compute the ribbonlength of a fold  and of an extended fold.

\begin{proposition}[\cite{Den-FRF} Corollary 18] \label{prop:extended-fold}
Using the notation from Figure~\ref{fig:AcuteVsObtuse}, we see that the ribbonlength
\begin{compactenum} 
\item  of a fold from angle $\theta$ is 
$$\Rib(|DC|+|CH|)= \frac{1}{\sin\theta};$$
\item of an extended fold from angle $\theta$ is 
$$\Rib(|EC|+|CI|)=\begin{cases} \cot(\frac{\theta}{2})
 & \text{ when $0<\theta\leq \frac{\pi}{2}$,} \\
\cot(\frac{\pi}{2}-\frac{\theta}{2})& \text{when $\frac{\pi}{2}\leq \theta< \pi$.}
\end{cases}
$$
\end{compactenum}
\end{proposition}

\begin{corollary}[\cite{Den-FRF} Corollary 19] \label{cor:triangles}
The minimum folded ribbonlength needed to create a fold (or an extended fold) is $1$, and occurs when $\theta=\frac{\pi}{2}$. In this case the fold (or extended fold) consists of two identical right isosceles triangles whose equal sides have length $w$. 
\end{corollary}

\section{Folded ribbon equivalence}\label{sect:ribbon-equivalence}
We have already seen that one way to define ribbon equivalence is to only track the knot type and ignore the folded ribbon. 
Another way to understand folded ribbon equivalence is to keep track of the folded ribbon as well as the knot type. We start by reviewing the discussions found in \cite{Den-FRS,  Den-FRF, DKTZ}.  

Since an oriented folded ribbon knot can be viewed as a framed knot, we can compute the linking number between the knot diagram and one boundary component.  Recall that the {\em linking number} is a link invariant used to determine the degree to which components of a link are entangled.  Given an oriented two-component link $L=A\cup B$, the linking number  $\Lk(A,B)$ is an integer which is defined to be one half the sum of $+1$ crossings and $-1$ crossings between $A$ and $B$ (see Figure~\ref{crossingvalue}). 

\begin{figure}[htbp]
\begin{center}
\begin{overpic}{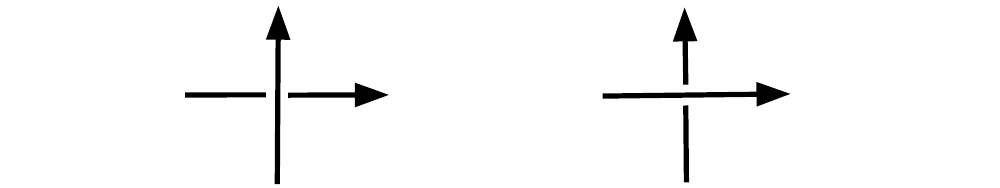}
\end{overpic}
\end{center}
\caption{The crossing on the left is labelled $-1$, the crossing on the right $+1$. Figure re-used with permission from~\cite{DKTZ}.}
\label{crossingvalue}
\end{figure}

 \begin{definition} [\cite{DKTZ} Definition 3.1] \label{ribbonlinkingnumber}
 Given an oriented folded ribbon knot $\kwf$, we define the {\em ribbon linking number} to be the linking number between the knot diagram and one boundary component of the ribbon.  We denote this as $\Lk(\kwf)$.
 \end{definition}
 
 \begin{definition} [\cite{Den-FRF} Definition 13] Two oriented folded ribbon knots (or links) are {\em (folded) ribbon equivalent} if 
 \begin{compactenum}
 \item they are diagram equivalent, 
 \item have the same ribbon linking number, and 
 \item are both topologically equivalent either to a M\"obius band or an annulus when considered as ribbons in $\R^3$.
 \end{compactenum}
\end{definition}

We remind the reader that if a polygonal knot diagram has an even number of edges, then the corresponding folded ribbon knot is a topological annulus. If a polygonal knot diagram has an odd number of edges, then the corresponding folded ribbon knot is a M\"obius band.

\begin{figure}[htbp]
\begin{center}
\begin{overpic}{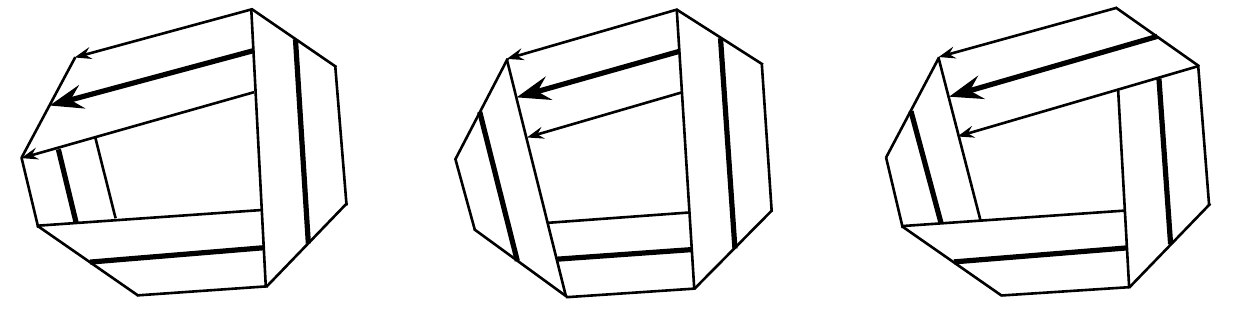}
\end{overpic}
 \caption{Three different folded ribbon 4-stick unknots. The left and center ribbon unknots have ribbon linking number 0, while the one on the right has ribbon linking number $-2$. Figure re-used with permission from \cite{DKTZ}. } 
\label{fig:4unknot}
\end{center}
\end{figure}

Some examples of ribbon equivalence are shown in Figure~\ref{fig:4unknot}. The left and center folded ribbon unknots are ribbon equivalent (with ribbon linking number 0), while the one on the right is not ribbon equivalent to them (with ribbon linking number $-2$). This example shows that there can be different looking folded ribbon knots with the same ribbon linking number. 
 
The rest of this paper is dedicated to finding folded ribbonlength with respect to {\em ribbon equivalence}.  We start by considering unknots. The $2$-stick unknot is a topological annulus with ribbon linking number 0. In this case, the ribbonlength of the unknot with respect to ribbon equivalence is also 0.  We will consider the folded ribbonlength of unknots with other ribbon linking numbers and which are M\"obius bands in the rest of this paper. Before we do that, we introduce some new notation to clearly describe these cases. 

\begin{definition}
Let  $\mk\in[\mk]$, and suppose that the corresponding folded ribbon knot $\kwf$ is of topological type $T$ and has ribbon linking number $\Lk(\kwf)=k$. Then, the {\em minimum folded ribbonlength} of such a folded ribbon knot $\kwf$ is defined to be
$$ \Rib([\kwf]) = \inf_{\substack{\kwf \in T,\\ \Lk(\kwf)=k}} \Rib(\kwf).$$
\end{definition}

In Section~\ref{sect:3-stick}, we find the minimal folded ribbonlength for 3-stick unknots whose corresponding folded ribbons are topological M\"obius bands. In Section~\ref{sect:n-stick}, we give upper bounds on folded ribbonlength of folded ribbon unknots with any ribbon linking number by considering unknots whose diagrams have obtuse interior angles.
In Section~\ref{sect:unknots}, we vastly improve these upper bounds by looking at unknot diagrams which are not regular.


\section{Computing ribbon linking number}\label{sect:linking-number}

Since we are interested in finding folded ribbonlength with respect to ribbon equivalence, we need to be able to easily find the ribbon linking number. This section contains a number of observations and shortcuts to help with this computation. In addition, we discuss the relationship between the ribbon linking number, twist, and writhe of folded ribbon knots. In Theorem~\ref{thm:RibLowerBd}, we use our knowledge of folds to give lower bounds on folded ribbonlength for folded ribbon knots with zero writhe. We end by showing that ribbon linking number is a complete invariant for folded ribbon unknots whose diagrams are non-degenerate convex $n$-gons. Specifically, in Proposition~\ref{prop:n-linking} we show that for convex $n$-gons, the ribbon linking number determines the folding information and vice versa. 

To start, we observe that the crossings between the ribbon's boundary and the knot diagram usually only happen in two places: at a crossing of the knot diagram, or at a fold.  Figure~\ref{fig:intersections-fc} illustrates some examples of these intersections. At a crossing of the knot diagram (Figure~\ref{fig:intersections-fc} right) we see that the linking number at the crossings between the boundary component(s) of the ribbon and the knot diagram are the same as the sign of the crossing of the knot diagram. At a fold, the sign of the crossings are also straightforward to compute, and we deduce the following lemma.

\begin{figure}[htbp]
\begin{center}
\begin{overpic}{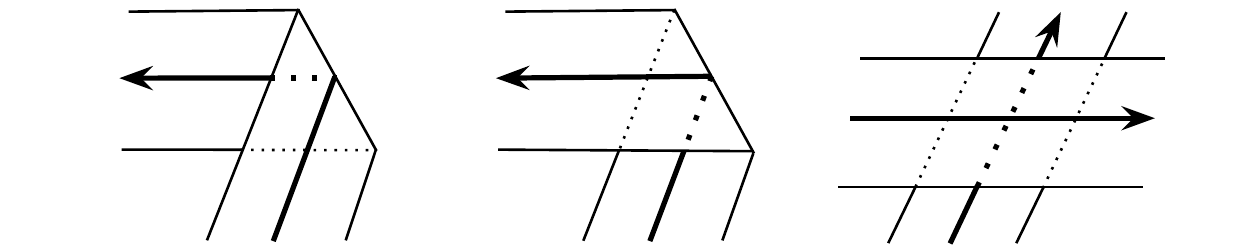}
\put(19.5,5.5){$+1$}
\put(18,14.5){$+1$}
\put(47.75,14.5){$-1$}
\put(49.5,5.5){$-1$}
\put(79,15.75){$+1$}
\put(70.5,7.75){$+1$}
\put(73,2.5){$+1$}
\put(85.5,7.75){$+1$}
\end{overpic}
 \caption{The values of the linking number between the ribbon's boundary and the knot diagram. From left to right: a left underfold, a left overfold, and a crossing. 
 } 
\label{fig:intersections-fc}
\end{center}
\end{figure}
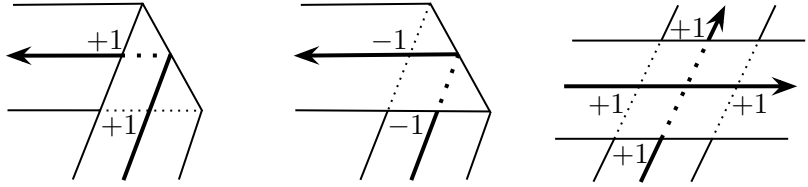

\begin{lemma} \label{lem:fold-compute}
Let $\kwf$ be an oriented folded ribbon knot. At a fold with fold angle $0<\theta< \pi$, there are two intersection points between the ribbon's boundary and the knot diagram. These intersection points both contribute the same sign to the ribbon linking number as follows:
\begin{compactitem}
\item left underfold $+1$,
\item left overfold $-1$,
\item right underfold $-1$,
\item right overfold $+1$.
\end{compactitem}

\end{lemma}

\begin{proof} Figure~\ref{fig:intersections-fc} illustrates the left underfold and left overfold case. By reversing the orientation of the folded ribbon in Figure~\ref{fig:intersections-fc}, we see the left underfold becomes a right overfold, and the left overfold becomes a right underfold. The signs of the crossings are unchanged by this reversal (since the orientation of both the knot diagram and boundary component are reversed).
\end{proof}

The restriction on the fold angle $\theta$ was explained previously in Remark~\ref{rmk:fold-angle}.  We follow up by noting that folds with fold angle $0$ do not contribute to the ribbon linking number because the ribbon boundary of the fold does not intersect the knot diagram. We saw this fact in action when we computed the ribbon linking number of the $2$-stick unknot. Recall that there are two folds with fold angle $0$ and so the ribbon linking number is $0$.

\begin{remark}\label{rmk:link-compute} The topological type of the ribbon affects the computation of ribbon linking number.  There are two boundary components when the folded ribbon is a topological annulus, and we only use one in our computation. Looking at Figure~\ref{fig:intersections-fc}, we see that at a fold, we only use one of the two crossings between the ribbon boundary and the knot diagram. At a crossing, we only use two of the four crossings between the ribbon boundary and knot diagram. Thus folds contribute $\pm \frac{1}{2}$ to the ribbon linking number and crossings contribute $\pm 1$. When the ribbon is a M\"obius band, there is a single boundary component and we use all crossings between the ribbon boundary and knot diagram. Thus folds contribute $\pm 1$ to the ribbon linking number and crossings contribute $\pm 2$. 
\end{remark}

For a given folded ribbon knot, we note that there may be additional places where the ribbon's boundary intersects the knot diagram of a different piece of ribbon without the corresponding knot diagrams intersecting. (For example, when there is a crossing close to a fold line, or on one side of a Reidemeister 2 move.) We won't explore such intersections here since they don't necessarily occur. Indeed, in \cite{Den-FRF} we prove that for any regular polygonal knot diagram $\mk$ and folding information $F$, there is a constant $M>0$ such that for all widths $w<M$, the corresponding folded ribbon knot $\kwf$ has only single and double points.

Later in Section~\ref{sect:unknots}, we construct unknot diagrams whose corresponding folded ribbon unknots are topological annuli and with specific ribbon linking number. The following observation is very useful there and elsewhere.
 
\begin{lemma} \label{lem:two-folds}
Suppose a folded ribbon knot $\kwf$ is a topological annulus. Consider a piece of ribbon of $\kwf$ which corresponds to part of a knot diagram with no self-intersections and two vertices. Then that piece of ribbon has ribbon linking number $\pm 1$ if and only if it has two folds with the same sign. 
\end{lemma}

\begin{figure}[htbp]
    \begin{overpic}{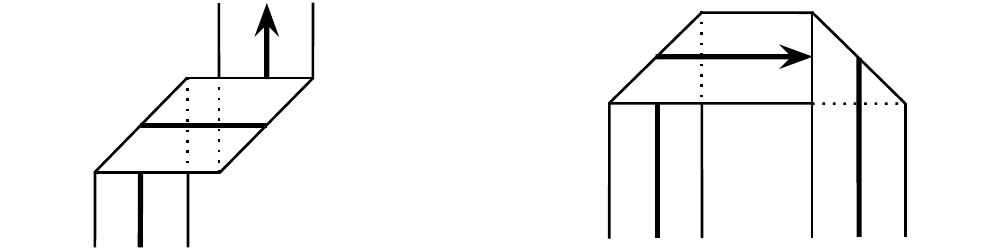}
    \put(-8,12){right overfold}
    \put(13,9){$+1$}
    \put(24,6){left underfold}
     \put(22.5,13.5){$+1$}
    \put(44,20){right overfold}
    \put(64,16){$+1$}
    \put(88,18){right overfold}
    \put(80.75,16){$+1$}
     \end{overpic}
    \caption{A piece of folded ribbon with zero writhe and ribbon linking number $+1$ must have two folds with the same sign.}
    \label{fig:linking-1}
\end{figure}
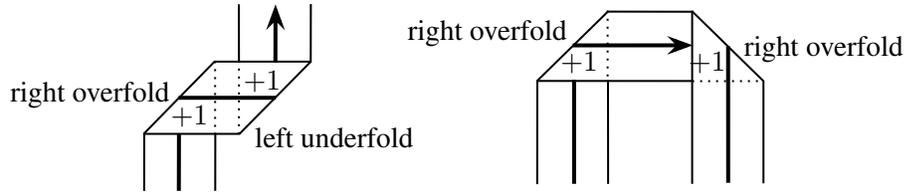

\begin{proof} Since we assume the knot diagram $\mk$ has no self-intersections, then the ribbon linking number depends only on the two folds (corresponding to the vertices of $\mk$). 
Recall that Lemma~\ref{lem:fold-compute} says that  the sign of a fold is as follows:  left underfold or right overfold~$+1$,  left overfold or right underfold~$-1$.  From Remark~\ref{rmk:link-compute}, we know that each fold contributes $\pm\frac{1}{2}$ to the ribbon linking number (since $\kwf$ is a topological annulus).
Thus, the two folds have the same sign if and only if the ribbon linking number is $\pm 1$. The two possible cases are shown in Figure~\ref{fig:linking-1}.
 \end{proof}

%

\subsection{Link, Twist and Writhe}\label{sect:LkTwWr}
Since oriented folded ribbon knots can be viewed as framed knots, we can also talk about the twist and writhe of the  ribbon.  In practical terms, the twist measures how much the ribbon rotates around its axis, while the writhe gives a measure of non-planarity of the ribbon's axis curve. 

\begin{definition}
Given an oriented folded ribbon knot $\kwf$,  we define the {\em twist} $\Tw(\kwf)$ to be one-half the sum of the crossing values between the knot diagram and one boundary component of the folded ribbon {\em computed only at folds}.  We define the {\em writhe} $\Wr(\kwf)$, to be the sum of the crossing values of its knot diagram $\mk$.  
\end{definition}

 We note that our definitions of twist and writhe exactly correspond to the definitions given in terms of integrals for ribbon knots in~$\mathbb{R}^3$ (see \cite{Adams, Egg}). The reader may wonder about non-regular knot diagrams for writhe; for example, the $2$-stick unknot. Here, we follow the integral definition of writhe, and we start with our unknot diagram in the plane. Then we can imagine bending one edge slightly above the plane in a direction normal to the plane, and bending one edge slightly below the plane. From any direction (other than the direction normal to the plane) the unknot will have writhe zero.  We conclude that this folded ribbon unknot has writhe $0$.

For ribbon knots in~$\R^3$, there is a well known relationship due to G.~C\u{a}lug\u{a}reanu~\cite{ Cal59, Cal61} and J.H.~White~\cite{Egg, Whi} between the linking number, twist, and the writhe. Since our definitions correspond to theirs, we immediately have the following:

\begin{theorem}\label{thm:link-twist-writhe}
Given an oriented folded ribbon knot $\kwf$, we have
\begin{enumerate}
\item  $\Lk(\kwf) =\Tw(\kwf)+\Wr(\kwf)$ \ \ when $\kwf$ is a topological annulus, 
\item $\Lk(\kwf) =\Tw(\kwf)+2\Wr(\kwf)$ \ \ when $\kwf$ is a M\"obius band.
\end{enumerate}
\end{theorem}

We can use our knowledge of the ribbonlength of folds to obtain a lower bound on the folded ribbonlength of unknots.

\begin{theorem} \label{thm:RibLowerBd}
Suppose the folded ribbon knot $\kwf$ has writhe $\Wr(\kwf) =0$ and ribbon linking number $\Lk(\kwf)=\pm n$. Then the minimum folded ribbonlength of such a folded ribbon knot is bounded from below as follows:
\begin{enumerate} 
\item $2n \leq \Rib([\kwf])$  \ \  when $\kwf$ is a topological annulus,
\item $n \leq \Rib([\kwf])$ \ \ when $\kwf$ is a M\"obius band. 
\end{enumerate} 
\end{theorem}

\begin{proof} Since $\Wr(\kwf)=0$, the ribbon linking number comes solely from the twist $\Tw(\kwf)$, that is from the folds of $\kwf$. We apply Remark~\ref{rmk:link-compute} and see that if $\kwf$ is an annulus, then there must be $2n$ folds with the same sign to give ribbon linking number $n$. If $\kwf$ is a M\"obius band, then there must be $n$ folds with the same sign to give ribbon linking number $n$.  In both cases, there may be additional folds or crossings. These folds and crossings will appear in pairs and have opposite signs. Corollary~\ref{cor:triangles} says that the minimum folded ribbonlength needed to create a fold is $1$.  Thus, the folded ribbonlength of $\kwf$ is the ribbonlength of the folds  plus any additional ribbonlength needed to create the knot. This immediately gives the lower bounds above.
\end{proof}

\subsection{Convex $n$-stick unknots}\label{sect:n-unknot}

We assume the unknot diagram $\um$ is a non-degenerate convex $n$-gon. (Non-degenerate means that the sidelengths are nonzero and the interior angles are not $\pi$. Convexity guarantees that the internal angles are not $0$.) The aim of this subsection is to see how the folding information affects the ribbon linking number and vice versa. 

\begin{figure}[htbp]
\begin{center}
\includegraphics{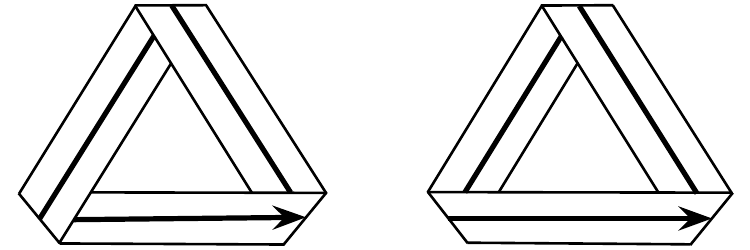}
\caption{On the left, a folded ribbon $3$-stick unknot where each fold is an underfold. On the right, a folded ribbon $3$-stick unknot where one fold is an overfold and the other two are underfolds. Figure re-used with permission from \cite{DKTZ}.}
\label{fig:3-stick-unknots}
\end{center}
\end{figure} 

We start by considering the simplest example: $3$-stick unknots.  There are two possibilities for the folding information of the corresponding folded ribbon unknot. The first is where the folds are all the same, and the second is where one fold is different from the other two.  Looking at Figure~\ref{fig:3-stick-unknots}, we see the folded ribbon $3$-stick unknots have been oriented in the counterclockwise direction. On the left, the folding information is under-under-under ($uuu$), and on the right, the folding information is under-under-over ($uuo$). A short computation using Lemma~\ref{lem:fold-compute} shows that the $uuu$ folded ribbon unknot has ribbon linking number $+3$, and the $uuo$ folded ribbon unknot has ribbon linking number $+1$. If we change the overfolds to underfolds and vice versa, then a similar calculation reveals the $ooo$ folded ribbon unknot has ribbon linking number $-3$, and the $oou$ folded ribbon unknot has ribbon linking number $-1$. Moreover, these are the only possibilities! 

\begin{remark}\label{rmk:fold-info3}  We have shown that ribbon linking number $\Lk(\uwf)$ is a complete invariant. Namely, for $3$-stick unknots oriented in the counterclockwise direction, we see that up to cyclic relabeling of the vertices:
\begin{compactenum}
\item   $Lk(\uwf)= + 3$ if and only if the folding information is $uuu$,
\item  $Lk(\uwf)= - 3$ if and only if the folding information is $ooo$,
\item   $Lk(\uwf)= + 1$ if and only if the folding information is $uuo$,
\item $Lk(\uwf)= - 1$, if and only if the folding information is $oou$.
\end{compactenum}

When $3$-stick unknots are oriented in the clockwise direction we get almost the same result, but the overfolds and underfolds are switched.  When we switch orientation, as we saw in the proof of Lemma~\ref{lem:fold-compute}, the left underfold becomes a right overfold and the left overfold becomes a right underfold. The signs of the crossings are unchanged by this reversal. 
\end{remark}

We now use this example as motivation for the general case when we have folded ribbon unknots $\uw$ whose knot diagrams are non-degenerate, convex $n$-gons for $n\geq 3$.  There are two cases to consider: either $n$ is odd and $\uw$ is a M\"obius band or $n$ is even and $\uw$ is an annulus. For example, Figure~\ref{fig:pentagon} shows folded ribbon unknots whose knot diagrams are a regular pentagon and a regular octagon. Their corresponding folded ribbon unknots are a M\"obius band and an annulus respectively.

\begin{figure}[htbp]
    \begin{overpic}{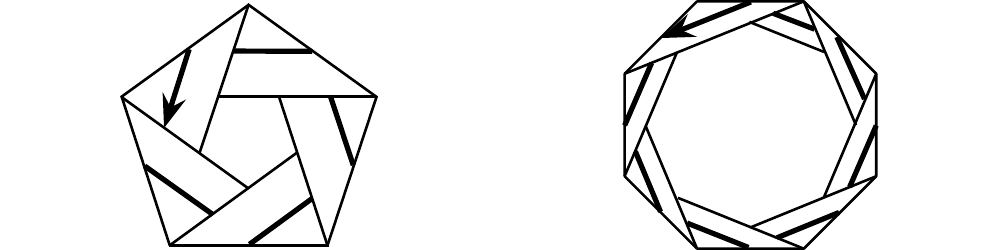}
    \end{overpic}
    \caption{Folded ribbons unknots corresponding whose knot diagrams are a regular pentagon (left) and a regular octagon (right).}
    \label{fig:pentagon}
\end{figure}

We start by working out the possible values of ribbon linking number for such unknots by recognizing that the only contribution to the ribbon linking number comes from the folds, then thinking about the possible folding information at each fold. Recall that we have the following shorthand for folds: $u$ represents an underfold and $o$ represents an overfold. We can repeat the arguments before Remark~\ref{rmk:fold-info3} to prove that the ribbon linking number of the corresponding folded ribbon unknot $\uwf$ is a complete invariant.

\begin{proposition} \label{prop:n-linking}
Let $\um$ be a non-degenerate convex $n$-gon for $n\geq 3$.  Then the corresponding folded ribbon  unknot $\uw$ can only have the following ribbon linking numbers:
\begin{compactenum}
\item when $n\geq 3$ and is odd,  then $\Lk(\uwf) = \pm 1, \pm 3, \dots , \pm n$;
\item when $n\geq 4$ and is even, then $\Lk(\uwf) = 0, \pm 1, \pm 2, \dots , \pm\frac{n}{2}$.
\end{compactenum}
Moreover, up to permutation among the vertices, the folding information determines the ribbon linking number and the ribbon linking number determines the folding information.
\end{proposition}
\begin{proof}
Without loss of generality, assume that $\um$ is oriented in a counterclockwise direction. Since the $n$-gon is convex, this means that $\um$ turns to the left at each vertex. (Just as in Remark~\ref{rmk:fold-info3}, when the orientation is reversed, we get the same result, but the overfolds and underfolds are switched.)

We start by assuming $n\geq 3$ is odd. Here, the corresponding folded ribbon unknot is a M\"obius band, and there is only one boundary component.  By using Lemma~\ref{lem:fold-compute}, we see that a left underfold contributes $+1$ and a left overfold contributes $-1$ to the ribbon linking number.  Suppose all the folds are underfolds, then the ribbon linking number is $+ n$. Suppose all the folds are overfolds, then the ribbon linking number is $-n$.   Now assume one fold is of a different type to the others: up to permutation, either $ou\dots u$, or $uo\dots o$. In either case, when computing the ribbon linking number we see a pair $+1, -1$ which sum to 0. Thus the ribbon linking number is $+ (n-2)$ for $ou\dots u$ and $-(n-2)$ for $uo\dots o$. If we assume there are two folds which are of a different type to the others, then when computing the ribbon linking number we have two pairs of $+1,-1$ which sum to 0, giving a ribbon linking number of $\pm(n-4)$. We can keep repeating this argument until there are $\lfloor\frac{n}{2}\rfloor$ of one kind of fold and $\lceil\frac{n}{2}\rceil$ of the other kind of fold. In this case, the ribbon linking number is $\pm 1$.  Thus for $n$ odd, the possible ribbon linking numbers are $\pm 1, \pm 3, \pm 5, \dots, \pm n$. 

Now assume $n\geq 4$ is even. Here, the corresponding folded ribbon unknot is a topological annulus, and there are two boundary components. Again using Lemma~\ref{lem:fold-compute}, we see that a left underfold contributes $+ \frac{1}{2}$ and a left overfold contributes  $-\frac{1}{2}$ to the ribbon linking number. Suppose all the folds are underfolds, then the ribbon linking number is  $+\frac{n}{2}$. Suppose all the folds are overfolds, then the ribbon linking number is $-\frac{n}{2}$.  (Recall $n$ is even, so $\pm \frac{n}{2}$ is an integer.) Now assume one fold is of a different type to the others: up to permutation, either $ou\dots u$, or $uo\dots o$. In either case, when computing the ribbon linking number we see a pair $+\frac{1}{2}, -\frac{1}{2}$ which sum to 0. Thus the ribbon linking number is $\pm (\frac{n}{2}-1)$. If we assume there are two folds which are of a different type to the others, then when computing the ribbon linking number we have two pairs of $+\frac{1}{2}, -\frac{1}{2}$ which sum to 0, giving a ribbon linking number of $\pm(\frac{n}{2}-2)$.  We can keep repeating this argument until there are $\frac{n}{2}$ of each kind of fold, and this gives ribbon linking number 0. Thus for $n$ even, the possible ribbon linking numbers are $0, \pm 1, \pm 2, \dots, \pm \frac{n}{2}$.

In all of these cases we see that the folding information determines the ribbon linking number and the ribbon linking number determines the folding information. In Appendix~\ref{append:n-linking}, we have listed the ribbon linking number with the corresponding folding information in complete detail. 
\end{proof}


\section{Ribbonlength of Folded ribbon $3$-stick unknots}\label{sect:3-stick}

One of the main goals of this paper is to understand folded ribbonlength of unknots with respect to ribbon equivalence. In this section we look at the simplest examples of unknots where the corresponding folded ribbon knot is a M\"obius band. In this case, the $3$-stick unknot diagram must be a non-degenerate triangle.

\begin{lemma}
A $3$-stick unknot diagram $\um$ is a non-degenerate triangle if and only if the corresponding folded ribbon unknot $\uwf$ is a M\"obius band.
\end{lemma}
\begin{proof}
Assume $\um$ is a non-degenerate triangle. Figure~\ref{fig:3-stick-unknots} shows the two kinds of corresponding folded ribbon unknots. In each case $\uwf$ is a M\"obius band.  Now, assume that $3$-stick unknot diagram is a degenerate triangle. Since we are assuming the knot diagram has 3 edges, this means they are all of nonzero length. Since the diagram is a degenerate triangle, this means that all three edges are collinear. Then one vertex $v$ is between the other two, thus the fold angle at $v$ is $\pi$ and there is no fold. The two edges adjacent to $v$ behave as one edge, and we see that $\uwf$ is a topological annulus.
\end{proof}

From this point on we assume that $\um$ is a non-degenerate triangle.  As we saw in Section~\ref{sect:n-unknot}, there are two possibilities for the folding information of the corresponding folded ribbon unknot: either the folds are all the same, or one fold is different from the other two.  In previous work~\cite{DKTZ}, we have made a first pass at understanding the folded ribbonlength of $3$-stick unknots. We computed the folded ribbonlength of a $3$-stick unknot whose knot diagram is an equilateral triangle, giving an upper bound on the folded ribbonlength.
\begin{proposition}[\cite{DKTZ}
]
\label{prop:3-stick-upper}
 The minimum ribbonlength of any folded ribbon $3$-stick unknot $\uwf$ is bounded above as follows:
\begin{enumerate}
\item $\Rib([\uwf])\leq 3\sqrt{3}$ \ \ when the ribbon linking number is $\Lk(\uwf)= \pm 3$,
\item $\Rib([\uwf])\leq \sqrt{3}$ \ \ when the ribbon linking number is  $\Lk(\uwf) = \pm 1$.
\end{enumerate}

\end{proposition} 

We fixed the width and in making this computation, we noticed that the structure of the folded ribbon dictated the folded ribbonlength. In the case where $\Lk(\uwf)=\pm 3$ and the folds are all of the same type, we can shrink the knot diagram until the ribbon's boundary meets in the interior of the knot diagram at the incenter of the equilateral triangle.\footnote{The incenter of a triangle is the intersection of the three interior angle bisectors of the triangle. This is also the center of the largest circle inscribed in a triangle, the incircle.} 

We made use of this observation and proved (see \cite{DKTZ} Theorem 5.1) for a folded ribbon $3$-stick unknot with ribbon linking number $\pm 3$ and width $w$, that $\frac{w}{2}\leq \rin$, where $\rin$ is the inradius of the knot diagram.\footnote{The inradius of a triangle is the radius of the incircle.} Note that this result means that in this setting,  $w\leq 2 \rin$, and we use this relationship to bound folded ribbonlength from below:
$$\frac{\Len(K)}{2\rin}\leq \Rib(\uwf).$$ We then combined this with the fact that amongst all triangles with the same perimeter the equilateral triangle has the maximum inradius, allowing us to prove the following:

\begin{theorem} [\cite{DKTZ}] \label{thm:3-stick-3}
 The minimum folded ribbonlength of any folded folded ribbon $3$-stick unknot $\uwf$ with ribbon linking number $\Lk(\uwf) = \pm 3$ is $\Rib([\uwf]) =3\sqrt{3}$.  This occurs when $\um$ is an equilateral triangle. 
\end{theorem}

A natural question to ask is whether the folded ribbonlength bound in Theorem~\ref{thm:3-stick-3}  is the smallest amongst all unknots whose corresponding folded ribbon knot has ribbon linking number $\pm 3$ and which are topological M\"obius bands. We have seen examples in other papers \cite{Den-FRF, KMRT} that increasing the number of edges in the knot diagram can lead to a decrease in the  folded ribbonlength.  However, we believe Theorem~\ref{thm:3-stick-3} does indeed give the minimum ribbonlength $\Rib(K) = 3\sqrt{3}$ for such folded ribbon unknots. This leads us to conjecture.

\begin{conjecture}
The minimum folded ribbonlength of any folded ribbon unknot $\uwf$ which is a M\"obius band with ribbon linking number $\Lk(\uwf) =\pm 3$ is 
$$\Rib([\uwf]) = \inf_{\substack{ \uwf \in \text{M\"ob},\\ \Lk(\uwf)=\pm 3}} \Rib(\uwf) = 3\sqrt 3.$$

\label{conj:3-stick-3}
\end{conjecture}


\subsection{Ribbonlength for folded ribbon $3$-stick unknots with $\Lk(\uwf)=\pm1$. }\label{sect:3stick}
Recall that the ribbon linking number is either $\pm1$, or $\pm 3$ for folded ribbon $3$-stick unknots. In this section we investigate the remaining case where  $\Lk(\uwf)=\pm1$.  By Proposition~\ref{prop:n-linking}, this means one fold is of different type to the other two.  The reader is encouraged to fold a strip of paper as they follow along with these arguments. If we fix the unknot diagram and try to increase the width of the ribbon, we see that the boundary of the ribbon does not intersect itself in the incenter of the knot diagram, unlike the $\Lk(\uwf)=\pm3$ case. Rather, the width of the ribbon can increase until the corners of the folds meet. We will use this principle to show that the maximum width occurs when the triangle is equilateral.

Before we state and prove a theorem, we need to set up some notation.   If we look at Figure~\ref{fig:3-stick-XYZ} below, we see the $3$-stick unknot is the thick black line and has vertices $A$, $B$, $C$, with corresponding fold angles $\alpha_1$, $\alpha_2$, and $\alpha_3$. We assume the width of the corresponding folded ribbon unknot is $w$. We have labeled the fold line at vertex $A$ as $A_1A_2$, the fold line at $B$ as $B_1B_2$, and the fold line at $C$ as $C_1C_2$. (These labels move in a counterclockwise direction around the figure.) We then extend the fold lines until they intersect with one another, labeling the intersection points $D,E,$ and $F$ as shown in Figure~\ref{fig:3-stick-XYZ}. We have indicated the exterior boundary of the folded ribbon by lines $A_2B_1$,   $B_2C_1$, $C_2A_1$,  and we denote the lengths of these lines by $x=|A_2B_1|$, $y=|B_2C_1|$, and $z=|C_2A_1|$.  To simplify our diagram, we have not shown the boundary of the folded ribbon that intersects the interior of triangle $ABC$ (that is, the lines $A_2C_1$, $B_2A_1$ and $C_2B_1$).

\begin{figure}[htbp]
    \centering
    \begin{overpic}[scale=1]{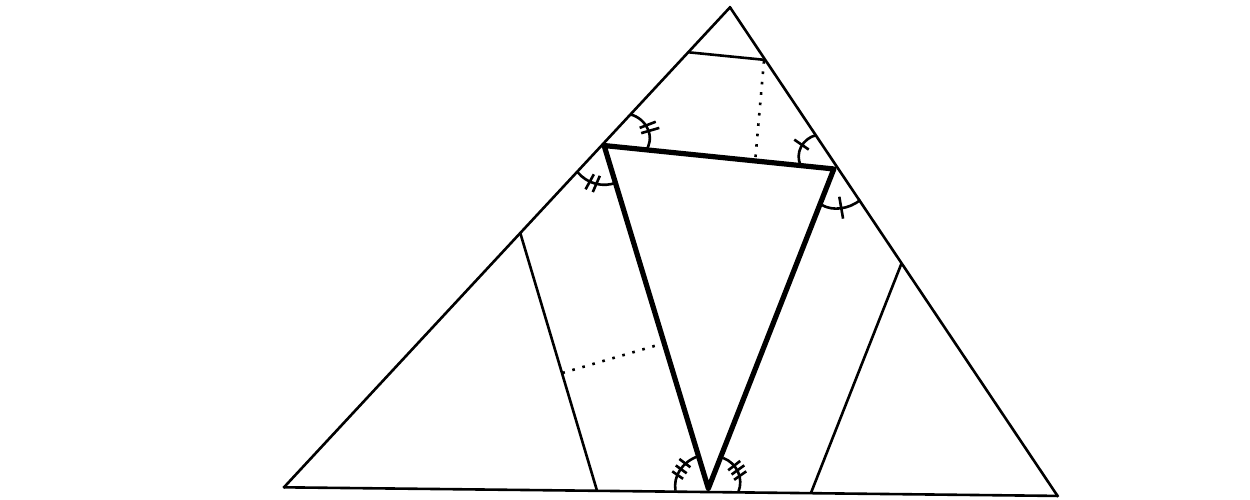}
    \put(67.5,27){$A$}
    \put(62.5, 24.5) {$\alpha_1$}
    \put(72, 19.75) {$A_1$}
    \put(62,34.5) {$A_2$}
    \put(59, 24){$G$}
    \put (58, 30) {$\frac{w}{2}$}
    \put(64.5, 18.5) {$\frac{\pi-\alpha_1}{2}$}
    \put(57, 33) {$x$}
    \put(57.5,40){$E$}
    \put(45.5,28.5){$B$}
    \put(51.25, 36) {$B_1$}
    \put(38,22) {$B_2$}
    \put(50, 25.5) {$\alpha_2$}
    \put(43.5, 21) {$\frac{\pi-\alpha_2}{2}$}
    \put(42, 11.5) {$y$}
    \put (48, 8) {$\frac{w}{2}$}
    \put(55.5,-2){$C$}
    \put(47,-2) {$C_1$}
    \put(63, -2.25) {$C_2$}
    \put(55.25, 7) {$\alpha_3$}
    \put(59, 3) {$\frac{\pi-\alpha_3}{2}$}
    \put(69, 8) {$z$}
    \put(20,0.5){$F$}
    \put(85,-0.5){$D$}
    \end{overpic}
    \caption{A 3-stick unknot $\um=\triangle ABC$, with the corresponding folded ribbon unknot having fold lines $A_1A_2$, $B_1B_2$ and $C_1C_2$. The lines containing the fold lines intersect at points $D, E, F$. The parts of the boundary of the folded ribbon unknot which meet outside of $\triangle ABC$ are $A_2B_1$,   $B_2C_1$, $C_2A_1$.}
    \label{fig:3-stick-XYZ}
\end{figure}
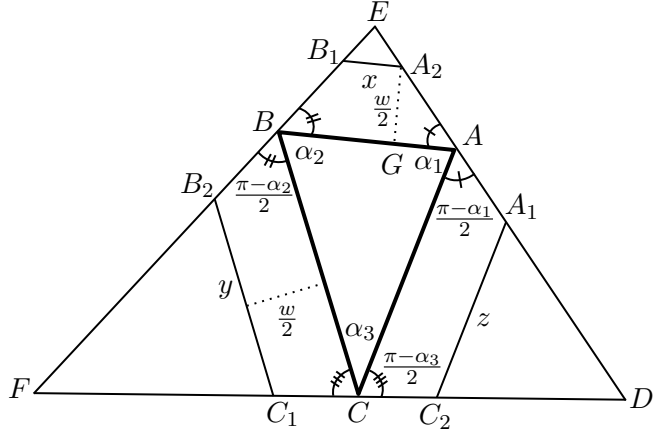

As described above, when we fix $\triangle ABC$, the maximum ribbon width occurs when the fold lines intersect in the exterior region of the knot diagram; that is, when one of the distances $x$, $y$ or $z$ is $0$. We can then fix the perimeter of the triangle $\triangle ABC$, and ask ourselves which angles give the maximum width. In answering this question we prove the following:

\begin{theorem} \label{thm:equilateral-3-stick}
Suppose $\uwf$ is a folded ribbon $3$-stick unknot with $\Lk(\uwf)=\pm1$ and assume $\Len(\um)=1$. Then the ribbon width is bounded from above by $w\le\frac{1}{\sqrt{3}}$, and $w$ achieves its maximum value when $\um$ is an equilateral triangle.
\end{theorem}

\begin{proof}
We will use the notation for the folded ribbon $3$-stick unknot  $\um=\triangle ABC$ given in Figure~\ref{fig:3-stick-XYZ}. We aim to maximize the width $w$ of the ribbon. We have two constraints: first $\alpha_1+\alpha_2+\alpha_3= \pi$ and second, $\Len(\um) =|AB|+|BC|+|CA|=1$. 

We first find $x=|A_2B_1|$, $y=|B_2C_1|$, and $z=|C_2A_1|$. In Figure~\ref{fig:3-stick-XYZ}, we know from the geometry of the fold that angle $\angle A_2AB=\angle A_1AC = \frac{\pi-\alpha_1}{2}$.  We then drop a perpendicular line segment $A_2G$ to $AB$.  Since we have fixed the width $w$,  then $|A_2G| = \frac{w}{2}$.
In right triangle $\triangle AGA_2$,  we know $\angle GA_2A = \frac{\alpha_1}{2}$, thus $|AG| =\frac{w}{2}\tan(\frac{\alpha_1}{2})$. Repeating these observations allows us to deduce that
\begin{align}
x&=|AB|-\frac{w}{2}\tan\frac{\alpha_1}{2}-\frac{w}{2}\tan\frac{\alpha_2}{2},\label{eq:edge1}\\ 
y&=|CB|-\frac{w}{2}\tan\frac{\alpha_2}{2}-\frac{w}{2}\tan\frac{\alpha_3}{2},\label{eq:edge2} \\
z&=|CA|-\frac{w}{2}\tan\frac{\alpha_3}{2}-\frac{w}{2}\tan\frac{\alpha_1}{2}.\label{eq:edge3}
\end{align}

Since $\Delta ABC$ is non-degenerate, we have $x< |AB|$, $y< |BC|$, $z< |CA|$, and so $0\leq x+y+z < 1$. Substituting Equations~\ref{eq:edge1}, \ref{eq:edge2} and \ref{eq:edge3} and rearranging, we find

 \begin{align*}
0&\le |AC|+|AB|+|BC| -2\w(\tan\Ai+\tan\Aii+\tan\Aiii) <1,\\
0&\le 1 -w(\tan\Ai+\tan\Aii+\tan\Aiii)<1,\\
-1&\le -w(\tan\Ai+\tan\Aii+\tan\Aiii) < 0.
\end{align*}
Dividing through by $-w$, we see that 
\begin{equation}
0 < \tan\Ai+\tan\Aii+\tan\Aiii\le\frac{1}{w}.
\label{eq:inequality}
\end{equation}

Thus to maximize the width, we need to  minimize the  function $f(\alpha_1,\alpha_2,\alpha_3):=\tan\Ai+\tan\Aii+\tan\Aiii$ subject to the constraint function $g(\alpha_1,\alpha_2,\alpha_3):=\alpha_1+\alpha_2+\alpha_3-\pi=0$. 
We use the method of Lagrange multipliers with the Lagrangian function \begin{align*}
\mathcal{L}(\alpha_1,\alpha_2,\alpha_3,\lambda),
&=f(\alpha_1,\alpha_2,\alpha_3)-\lambda g(\alpha_1,\alpha_2,\alpha_3).
\end{align*}
We set the gradient of the Lagrangian function equal to 0 and see that
\begin{align*}
0= \nabla\mathcal{L}(\alpha_1,\alpha_2,\alpha_3,\lambda)
&=\nabla f(\alpha_1,\alpha_2,\alpha_3)-\lambda\nabla g(\alpha_1,\alpha_2,\alpha_3),\\
0&= \nabla(\tan\Ai+\tan\Aii+\tan\Aiii)-\lambda\nabla(\alpha_1+\alpha_2+\alpha_3-\pi),\\
0&= \langle \frac{1}{2}\sec^2\Ai, \frac{1}{2}\sec^2\Aii, \frac{1}{2}\sec^2\Aiii\rangle - \lambda\langle1,1,1\rangle.
\end{align*}
We then set up a system of equations as follows:
\begin{align}
\lambda &= \frac{1}{2}\sec^2\Ai,  \label{eq:a1} \\
\lambda &= \frac{1}{2}\sec^2\Aii,   \label{eq:a2} \\
\lambda &= \frac{1}{2}\sec^2\Aiii,  \label{eq:a3} \\
\pi&=\alpha_1+\alpha_2+\alpha_3.   \label{eq:sum}
\end{align}
Equations~\ref{eq:a1} , \ref{eq:a2}, and \ref{eq:a3}  yield
$$\sec^2\Ai = \sec^2\Aii = \sec^2\Ai,$$
or,
\begin{equation}\pm\cos\Ai = \pm\cos\Aii = \pm\cos\Aiii.\label{eq:equal}
\end{equation}
We recall that $\triangle ABC$ is non-degenerate, and so $0< \alpha_i<\pi$ for all $i$. This means $\cos\alpha_i=\cos\alpha_j$ has just one solution: $\alpha_i = \alpha_j$. We also recall the trigonometric identity $-\cos(\theta)=\cos(\pi-\theta)$. We use these two facts to reduce the number of cases we need to consider when solving Equations~\ref{eq:sum} and \ref{eq:equal}.

\begin{enumerate}
    \item[Case~1.]  Assume that $\cos\Ai = \cos\Aii = \cos\Aiii$. We see $\alpha_1=\alpha_2=\alpha_3$, and combining this with Equation~\ref{eq:sum} gives
    \begin{align*}
    \pi&=\alpha_1+\alpha_2+\alpha_3=3\alpha_1,   \\
    \frac{\pi}{3}&=\alpha_1=\alpha_2=\alpha_3.
    \end{align*}
    \item[Case~2.] Without loss of generality, assume that  $-\cos\Ai = \cos\Aii = \cos\Aiii$. Then rewriting, $$\cos(\pi-\frac{\alpha_1}{2}) = \cos\Aii = \cos\Aiii.$$ 
    We then see that $2\pi-\alpha_1=\alpha_2=\alpha_3$, and so $2\pi=\alpha_1+\alpha_2$. However, we assumed that $0<\alpha_1,\alpha_2<\pi$ which means $0<\alpha_1+\alpha_2<2\pi$, giving a contradiction.

     \item[Case 3.] Without loss of generality, assume that  $-\cos\Ai = -\cos\Aii = \cos\Aiii$. Then rewriting, $$\cos(\pi-\frac{\alpha_1}{2}) = \cos(\pi-\Aii) = \cos\Aiii.$$ 
    We then see that $2\pi-\alpha_1=2\pi-\alpha_2=\alpha_3$, and so $2\pi=\alpha_2+\alpha_3$.  As in Case~2, this is a contradiction to our assumption about $\alpha_2$ and $\alpha_3$.
    \end{enumerate}

To conclude, we have found that $\frac{\pi}{3}=\alpha_1=\alpha_2=\alpha_3$ is the only extreme value of our function $f$ subject to the constraint $g$.  To check whether this solution yields a minimum or maximum value for $f(\alpha_1, \alpha_2, \alpha_3)$, we compare values of the function at a nearby point. We compute
\begin{equation*}
    f(\frac{\pi}{6},\frac{\pi}{3},\frac{\pi}{2})=2-\sqrt{3} + \frac{1}{\sqrt{3}} +1 = 3-\frac{2\sqrt{3}}{3}>\sqrt{3}=f(\frac{\pi}{3}, \frac{\pi}{3}, \frac{\pi}{3}). 
    \label{eq:solution}    
    \end{equation*} 
    
Thus, the solution $\frac{\pi}{3}=\alpha_1=\alpha_2=\alpha_3$ minimizes $f$. That is, the knot diagram $\um$ is an equiangular triangle. Using the Law of Sines, we see that $\um$ must be an equilateral triangle. When we substitute this result into Equation~\ref{eq:inequality}, we see that 
$$ \sqrt{3}\le\tan\Ai+\tan\Aii+\tan\Aiii\le\frac{1}{w}.\\
$$
In summary, the ribbon width is bounded from above by $w\leq \frac{1}{\sqrt{3}}$, and ribbon width achieves its maximum value when $\um$ is an equilateral triangle.
\end{proof}

\begin{corollary}\label{cor:3stickbound}
The minimum folded ribbonlength of any folded ribbon $3$-stick unknot with ribbon linking number $\Lk(\uwf)=\pm1$ is bounded from below by $\sqrt{3}\le\Rib([\uwf])$. The minimum value is achieved when $\um$ is an equilateral triangle.
\end{corollary}
\begin{proof} Since folded ribbonlength is a scale-invariant quantity, we set $\Len(\um)=1$ and apply the results of Theorem~\ref{thm:equilateral-3-stick}. Then
$$\sqrt{3}\le\frac{1}{w}=\frac{\Len(K)}{w}=\Rib(\uwf).$$ 
\end{proof}

\begin{corollary}\label{cor:3-stick-1}
The minimum folded ribbonlength of any folded ribbon $3$-stick unknot $\uwf$ with ribbon linking number $\Lk(\uwf)=\pm1$ is $\Rib([\uwf])=\sqrt{3}$. This occurs when $\um$ is an equilateral triangle.
\end{corollary}

\begin{proof} Proposition~\ref{prop:3-stick-upper} and Corollary~\ref{cor:3stickbound} allow us to bound the folded ribbonlength above and below by $\sqrt{3}$. These results also show the minimum occurs when $\um$ is an equilateral triangle. 
\end{proof}

Rather than just restricting our attention to folded ribbon $3$-stick unknots, we can minimize folded ribbonlength amongst all folded ribbon unknots which are M\"obius bands with ribbon linking number $\pm 1$. We conjecture that Corollary~\ref{cor:3-stick-1} still gives the minimum folded ribbonlength.

\begin{conjecture} The minimum folded ribbonlength of any folded ribbon unknot $\uwf$ which is a M\"obius band with ribbon linking number $\Lk(\uwf)=\pm 1$ is
$$\Rib([\uwf]) =  \inf_{\substack{ \uwf \in \text{M\"ob},\\ \Lk(\uwf)=\pm 1}} = \sqrt{3}.$$
\label{conj:3-stick-1}
\end{conjecture}

\section{Ribbonlength for Folded Ribbon Convex $n$-Stick Unknots}\label{sect:n-stick}

In the previous section we looked at folded ribbon $3$-stick unknots. We completely determined the minimum folded ribbonlength for these unknots. What about other unknots made of more sticks? We will restrict our attention to a very particular type: folded ribbon unknots whose knot diagrams are non-degenerate, convex $n$-gons.  For example, previously in Figure~\ref{fig:pentagon}, we saw folded ribbon unknots whose knot diagrams are a regular pentagon and a regular octagon.  
In previous work  \cite{DKTZ}, we computed folded ribbon length for regular $n$-gons for $n\geq 4$. 

\begin{proposition}[\cite{DKTZ}] \label{prop:n-stick-upper} 
The minimum folded ribbonlength of any folded ribbon $n$-stick unknot $\uwf$ for $n\geq 4$ is bounded above by $\Rib([\um])\leq n \cot(\frac{\pi}{n})$. 
\end{proposition} 

To see this we recall that a regular $n$-gon has folded ribbonlength $n\cot(\frac{\pi}{n})$ (\cite{DKTZ} Proposition~4.3). In our computation we observed that the ribbon width was maximized when the fold lines met on the exterior of the $n$-gon, as shown on the left in Figure~\ref{fig:obtuse-angles} for a regular pentagon.  In the next section, we show that this observation is true not just for regular $n$-gons, but for any $n$-gon with obtuse interior angles.

\subsection{Unknots with obtuse interior angles}

In this section, we assume the unknot diagram $\um$ is a non-degenerate $n$-gon with $n\geq 4$ and with {\em obtuse interior angles}. As before, non-degenerate means that the side lengths are nonzero and the interior angles are not $\pi$. The assumption about interior angles means that {\em the $n$-gon is convex}!   In this special case, we fix the length of $\um$ and try and maximize the width of the corresponding folded ribbon unknot. The assumption on the angles of $\um$ mean that the fold angles of $\uwf$ are all obtuse.

As we recall from Section~\ref{sect:useful-geom}, this assumption has an effect on the geometry of the folds.  When two adjacent angles are obtuse, as on the right in Figure~\ref{fig:obtuse-angles},  we can increase the ribbon width until the fold lines meet. When the interior angles are acute, other factors can come into play, as discussed after Proposition~\ref{prop:3-stick-upper} in Section~\ref{sect:3-stick}.

\begin{figure}[htbp]
    \begin{overpic}{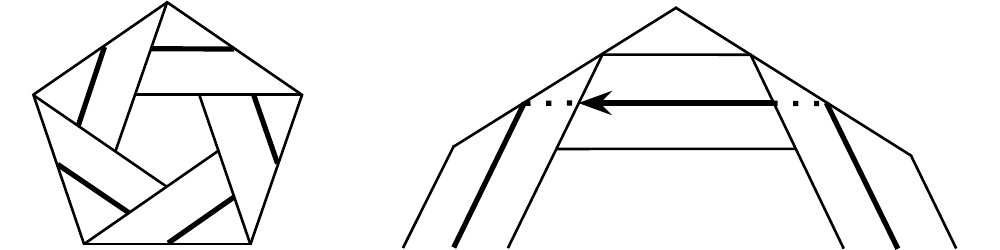}
     \put(83,15){$v_1$}
     \put(79,12){$\alpha_1$}
     \put(67,17){$x_1$}
     \put(65, 25){$w_{1,2}$}
     %
     \put(48.75,15.5){$v_2$}
     \put(52,12){$\alpha_2$}
     \put(48.5, 4){$x_2$}
     \put(89,4){$x_n$}
    \end{overpic}
    \caption{On the left, a folded ribbon unknot whose knot diagram is a regular pentagon. On the right, part of a folded ribbon $n$-stick unknot showing that the width can be increased until two folds meet at $w_{1,2}$.}
    \label{fig:obtuse-angles}
\end{figure}
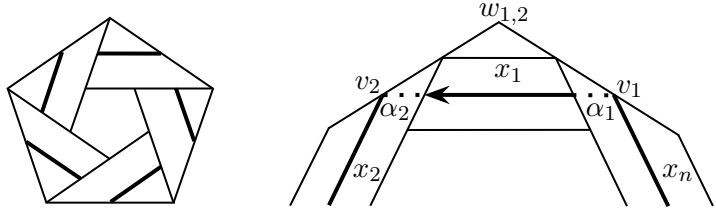

\begin{lemma}  \label{lem:fold-meet}
Suppose $\um$ is a non-degenerate $n$-gon whose interior angles satisfy $\frac{\pi}{2}\leq \alpha_i<\pi$. Then the supremum of the width of the corresponding folded ribbon unknot $\uwf$ occurs when a pair of adjacent fold lines of $\uwf$ meet in the region exterior to $\um$.
\end{lemma}
\begin{proof}
We refer to the notation in Figure~\ref{fig:AcuteVsObtuse1} and make two observations.  First, an extended fold consists of two parts: one double layer of ribbon in the fold $\triangle AFB$, and two single layers of ribbon in triangles $\triangle AFG$ and $\triangle BFJ$. Secondly,  when the fold angle $\theta$ is obtuse, the triangles $\triangle AFG$ and $\triangle BFJ$ share part of their boundary ($GF$ and $FJ$) with the ribbon's boundary on the same side as the fold angle. When the fold angle $\theta$ is acute, the triangles $\triangle AFG$ and $\triangle BFJ$ share part of their boundary  ($AG$ and $BJ$) with the ribbon's boundary on the opposite side of the fold angle. 

\begin{center}
\begin{figure}[htbp]
\begin{overpic}[scale=0.9]{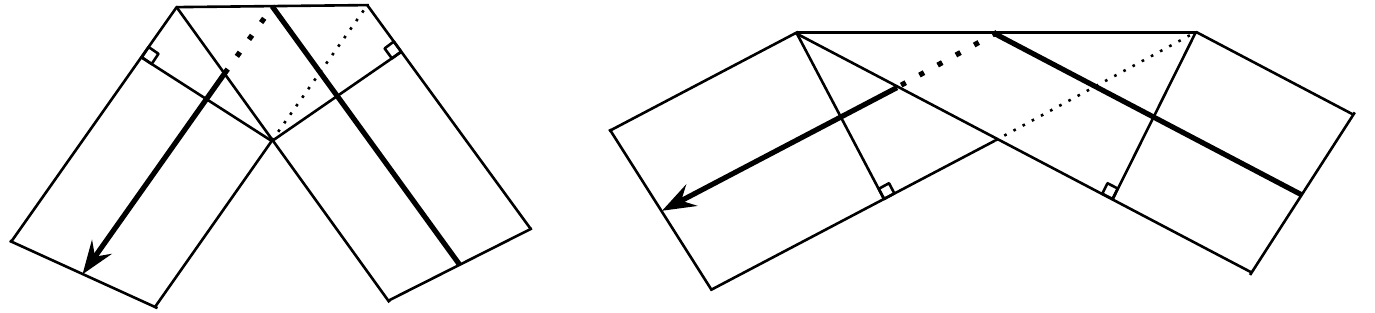}
\put(12,23){$A$}
\put(19,23){$C$}
\put(19.15, 19){$\theta$}
\put(26,23){$B$}
\put(17,16.5){$D$}
\put(11.5,14){$E$}
\put(18.5,9){$F$}
\put(8,18.5){$G$}
\put(20,16.5){$H$}
\put(23.5, 13){$I$}
\put(29.5, 19){$J$}
\put(57.5,21){$A$}
\put(72,21){$C$}
\put(71.5, 17.5){$\theta$}
\put(86,21){$B$}
\put(67.5,15.5){$D$}
\put(57.5,14){$E$}
\put(71.5,9.5){$F$}
\put(64,6.25){$G$}
\put(78,13.5){$H$}
\put(83.4,11){$I$}
\put(79, 6){$J$}
\end{overpic}
\caption{The extended fold from angle $\theta$ is the two pieces of ribbon determined by $EC+CI$. Figure re-used with permission from \cite{Den-FRC}.}
\label{fig:AcuteVsObtuse1}
\end{figure}
\end{center}

The unknot diagram $\um$ divides the plane into a bounded and unbounded region. The latter is the region exterior to $\um$. The assumption that the interior angles of $\um$ are obtuse and our observations in the previous paragraph allows us to deduce that
\begin{compactenum}
\item the fold angles are the interior angles of $\um$, 
\item the fold lines occur in the region exterior to $\um$,
\end{compactenum}
From this we conclude that the supremum of width occurs when a pair of adjacent fold lines meet.
\end{proof}

\begin{remark}\label{rmk:obtuse-only}
The fact the fold angles are obtuse is the only factor affecting the width of the ribbon. This means that the folding information and hence ribbon linking number (see  Proposition~\ref{prop:n-linking}) has no effect on the width. Thus any bounds on folded ribbonlength will apply to folded ribbons whose ribbon linking number is determined by Proposition~\ref{prop:n-linking}. We expand upon these ideas in Section~\ref{sect:ngon-summary}.
\end{remark}

Just as in Section~\ref{sect:3stick}, we can fix the length of the unknot diagram $\um$ and ask which angles and sidelengths give the maximum width. Lemma~\ref{lem:fold-meet} guarantees that the fold lines meet in the region exterior to $\um$. This allows us to use similar arguments as Theorem~\ref{thm:equilateral-3-stick} to first prove that the width is maximized when $\um$ is equiangular.

\begin{proposition}\label{prop:equiangular}
Suppose $\um$ is a non-degenerate $n$-gon for $n\geq 4$ whose  interior angles satisfy $\frac{\pi}{2}\leq \alpha_i<\pi$. Then the corresponding folded ribbon unknot $\uwf$ has ribbon width bounded from above by $w\leq \frac{1}{n\cot(\frac{\pi}{n})}$, and $w$ is maximized when $\um$ is equiangular.
\end{proposition}

\begin{proof} Since the proof of this proposition is almost identical to the proof of Theorem~\ref{thm:equilateral-3-stick}, we have moved it to Appendix~\ref{append:equiangular}.
\end{proof}

\begin{figure}[htbp]
    \begin{overpic}{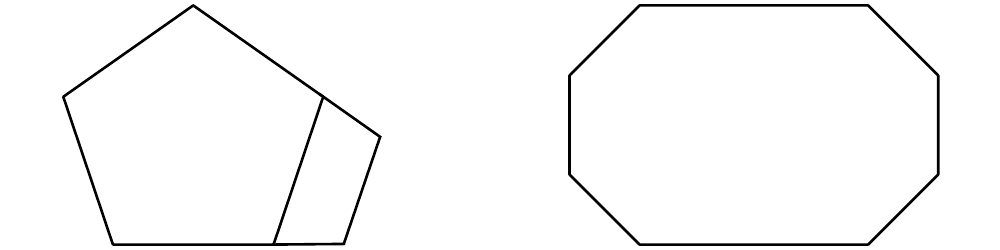}
       \put(18,1.5){$1$}
       \put(27,6.5){$2$}
       \put(37,4.5){$2'$}
       \put(24,17){$3$}
       \put(75,1.5){$1$}
       \put(88,4.5){$2$}
       \put(91.5,11){$3$}
       \put(88,18.5){$4$}
       \put(75,21){$5$}
       \put(60.5,18.5){$6$}
       \put(58,11){$7$}
       \put(60.5,4.5){$8$}
    \end{overpic}
    \caption{For all $n\geq 4$, equiangular $n$-gons are not necessarily equilateral.}
    \label{fig:polygons}
\end{figure}
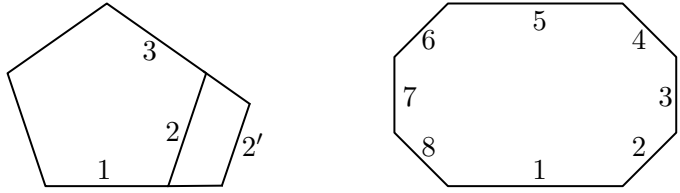

We note that an equiangular $n$-gon is equilateral only when $n=3$. To see this, take any regular $n$-gon and consider two cases: when $n$ is odd or even. For $n\geq 5$ and odd, we take the first and third side of the $n$-gon and extend them by the same amount. This is shown on the left in Figure~\ref{fig:polygons} for a pentagon. Then, construct a new second side parallel to the original second side. The new $n$-gon is equiangular, but not equilateral. For $n\geq4$ and even, then the $n$-gon has two sides which are parallel to each other. If we extend these two sides by the same amount, then the resulting $n$-gon is equiangular, but not equilateral. This is shown for the $8$-gon on the right in Figure~\ref{fig:polygons}.

We now show that amongst all equiangular $n$-gons, the one which is equilateral has the smallest folded ribbonlength.

\begin{theorem}\label{thm:equilateral-n-stick}
 The minimum folded ribbonlength of any folded ribbon $n$-stick unknot with $n\geq 4$ and whose fold angles $\alpha_i$ satisfy $\frac{\pi}{2}\leq \alpha_i<\pi$ is bounded from below by $n\cot(\frac{\pi}{n})\leq \Rib([\uwf])$. The minimum occurs when $\um$ is a regular $n$-gon.
\end{theorem}

\begin{figure}[htbp]
    \begin{overpic}{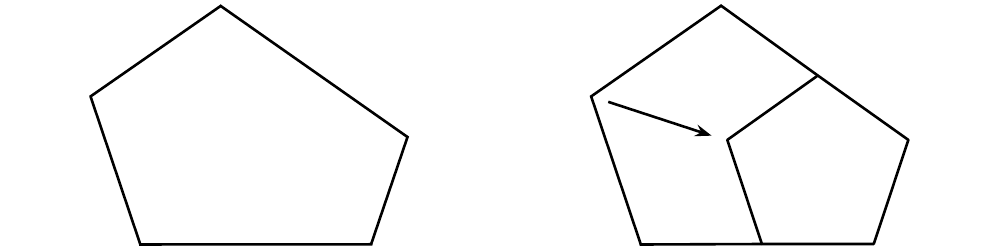}
    \put(38,0){$v_1$}
    \put(41.5,10){$v_2$}
    \put(21,25.5){$v_3$}
    \put(5,15){$v_4$}
    \put(10,0){$v_5$}
     \end{overpic}
    \caption{Equilateral $n$-gons have the shortest perimeter amongst all equiangular $n$-gons.}
    \label{fig:polygons-rescale}
\end{figure}
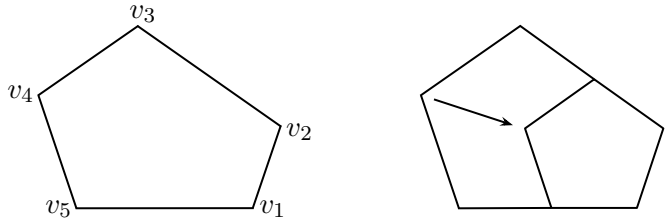

\begin{proof} 
From Proposition~\ref{prop:equiangular}, we know that the maximum width occurs when $\um$ is equiangular and each interior angle (fold angle) $\alpha_i=(\frac{n-2}{n})\pi$.  Suppose $\um$ is equiangular, but not equilateral. Also assume that the width is fixed at the maximum $w=\frac{1}{n\cot(\frac{\pi}{n})}$.   

Since $\um$ is not equilateral, then there is a shortest edge. Without loss of generality suppose the shortest edge is $v_1v_2$ (see Figure~\ref{fig:polygons-rescale} left). Referring to the notation in the proof of Proposition~\ref{prop:equiangular}, we know the fold lines at $v_1$ and $v_2$ meet, and the distance $x_1=0$. This edge determines the maximum ribbon width for $\uwf$.

Since $\um$ is not equilateral, there is some other edge which is longer than $v_1v_2$. We can replace $\um$ with a regular $n$-gon $\um'$ whose edges all have the same length as $v_1v_2$. This means that every pair of adjacent fold lines of $\mathcal{U}_{w,F}'$ meet, and we cannot decrease the length of $\um'$ any further. The width of the corresponding folded ribbon unknot $\uwf'$ is still $w=\frac{1}{n\cot(\frac{\pi}{n})}$.  Since the widths of $\uwf$ and $\uwf'$  are the same, but $\Len(\um')\leq \Len(\um)$, we have decreased the folded ribbonlength. Thus, the minimum folded ribbonlength occurs when $\um$ is a regular $n$-gon. If we then set the length of the equilateral $n$-gon to be $\Len(\um')=1$, we see that 
$$n\cot(\frac{\pi}{n})\le\frac{1}{w}=\frac{\Len(\um')}{w}=\Rib(\uwf').$$ 
\end{proof}

\begin{corollary}\label{cor:minngon}
The minimum folded ribbonlength for any folded ribbon $n$-stick unknot  $\uwf$ with $n\geq 4$ and whose fold angles satisfy $\frac{\pi}{2}\leq \alpha_i<\pi$ is $\Rib([\uwf])= n\cot(\frac{\pi}{n})$. This occurs when $\um$ is a regular $n$-gon.

\end{corollary}
\begin{proof} Using Proposition~\ref{prop:n-stick-upper} and Theorem~\ref{thm:equilateral-n-stick} we see that amongst all such folded ribbon unknots, the minimum folded ribbonlength is bounded above and below by $n\cot(\frac{\pi}{n})$. These bounds occur when $\um$ is a regular $n$-gon. 
\end{proof}

The assumption that the fold angles are obtuse was key in proving  Proposition~\ref{prop:equiangular} and hence Corollary~\ref{cor:minngon}. We can loosen the requirement that the fold angles are obtuse and instead require that the folded ribbon $n$-stick unknot be convex. We conjecture that the minimal ribbonlength of any folded ribbon convex $n$-stick unknot is still  $\Rib([\uwf])=n\cot(\frac{\pi}{n})$.

\begin{conjecture}
The minimum folded ribbonlength of any folded ribbon convex $n$-stick unknot with $n\geq 4$ is $\Rib([\uwf]) = n\cot(\frac{\pi}{n})$. This occurs when $\um$ is a regular $n$-gon.
\end{conjecture}
What about non-convex knots? We will show in Section~\ref{sect:unknots} that there are unknots with non-convex and  irregular knot diagrams whose corresponding folded ribbons have folded ribbonlength that is much smaller than that of Corollary~\ref{cor:minngon}.

\subsection{Unknots, ribbon equivalence, and folded ribbonlength}\label{sect:ngon-summary}
We know for a regular $n$-gon, the corresponding folded ribbon unknot $\uwf$ has folded ribbonlength $\Rib(\uwf)=n\cot(\frac{\pi}{n})$. We also saw that Proposition~\ref{prop:n-linking} gives us all possibilities for ribbon linking numbers of convex (and hence regular) $n$-gons. These two observations allow us to give upper bounds on folded ribbonlength for unknots in different ribbon equivalence classes. 

\begin{theorem}\label{thm:n-unknot1}
Suppose $\um$ is an unknot whose corresponding folded ribbon knot  $\uwf$ is a M\"obius band.
Then the minimum folded ribbonlength of any such unknot is bounded above by
\begin{compactenum} 
\item $\Rib([\uwf]) \leq \sqrt{3}$ when  $\Lk(\uwf)=\pm 1$,
\item $\Rib([\uwf]) \leq n\cot(\frac{\pi}{n})$ when $\Lk(\uwf)=\pm n$ and $n\geq 3$ is odd.
\end{compactenum}
Suppose $\um$ is an unknot whose corresponding folded ribbon knot $\uwf$ is a topological annulus. Then the minimum folded ribbonlength of any such unknot is bounded above by
\begin{compactenum}
\item $\Rib([\uwf])\leq 4$ when $\Lk(\uwf)=\pm 1$,
\item $\Rib([\uwf])\leq 2n\cot(\frac{\pi}{2n})$ when $\Lk(\uwf)=\pm n$ and $n\geq 2$.
\end{compactenum}
\end{theorem}

\begin{proof} 
In each case, we will give an example where the unknot $\um$ is a regular $n$-gon and the corresponding folded ribbon unknot has the desired ribbon linking number. This example will give the upper bound on the minimal folded ribbonlength.  In each case, we use the fact that the folded ribbonlength is $n\cot(\frac{\pi}{n})$ when $\um$ is a regular $n$-gon. We also use the fact than when $n$ is odd the corresponding folded ribbon unknot is a M\"obius band, and when $n$ is even it is a topological annulus. 

First, assume  that $\uwf$ is a M\"obius strip. Recall that when $\um$ is an equilateral triangle
\begin{compactitem}
\item  $\Rib(\uwf)=\sqrt{3}$ when $\Lk(\uwf)=\pm 1$ (Corollary~\ref{cor:3-stick-1}), and 
\item $\Rib(\uwf)=3\sqrt{3}$ when $\Lk(\uwf)=\pm 3$ (Theorem~\ref{thm:3-stick-3}). 
\end{compactitem}
We also note that $3\sqrt{3}=3\cot(\frac{\pi}{3})$. Now, suppose $n\geq 5$ and odd, we choose $\um$ to be a regular $n$-gon.  Then $\Rib(\uwf)= n\cot(\frac{\pi}{n})$ and by Proposition~\ref{prop:n-linking}, we can choose the folding information so that $\Lk(\uwf)=\pm n$. 
 
 Second, assume that $\uwf$ is a topological annulus. We know that when the unknot has two edges,  the corresponding folded ribbon knot $\uwf$ has $\Lk(\uwf)= 0$. Thus we take $\um$ to be a square (a regular 4-gon) for $\Lk(\uwf)=\pm 1$. We use Proposition~\ref{prop:n-linking} to choose the folding information so that $\Lk(\uwf)=\pm 1$. Then $\Rib(\uwf)= 4\cot(\frac{\pi}{4})=4$. Now suppose that $n\geq 2$ and that $\um$ is a regular $2n$-gon. Again using Proposition~\ref{prop:n-linking}, we choose the folding information so that $\Lk(\uwf)=\pm n$. Then $\Rib(\uwf)= 2n\cot(\frac{\pi}{2n}).$
\end{proof}

\section{Unknots, topological annuli, and linking number}\label{sect:unknots}

In this section, we will work exclusively with folded ribbon unknots $\uwf$ that are topological annuli. In other words, their unknot diagrams $\um$ have an even number of edges. We have already seen that when the ribbon linking number $\Lk(\uwf)=0$, the minimum folded ribbonlength of such folded ribbon unknots is $\Rib([\uwf])=0$.  What about folded ribbon unknots with nonzero ribbon linking number?

We can use Lemma~\ref{lem:two-folds} to build a folded ribbon unknot with the desired ribbon linking number. We consider a piece of ribbon of $\uwf$ that corresponds to part of a knot diagram with no self-intersections and two vertices. Then Lemma~\ref{lem:two-folds} says that piece of ribbon has ribbon linking number $\pm 1$ if and only if it has two folds with the same sign (either $+1$ or $-1$). We will also use Corollary~\ref{cor:triangles}, which tells us that the minimum folded ribbonlength needed to create a fold is $1$ and occurs when the fold angle $\theta=\frac{\pi}{2}$. In this case, the fold consists of two identical right isosceles triangles whose equal sides have length $w$, the width of the folded ribbon knot.

\begin{proposition} \label{prop:linking-1} There is a folded ribbon unknot $\uwf$ which is a topological annulus, such that 
\begin{compactenum}
\item the knot diagram $\um$ has four sticks,
\item $\uwf$ has ribbon linking number $\Lk(\uwf)=\pm 1$, and 
\item $\uwf$ has folded ribbonlength $\Rib(\uwf)= 2$.
\end{compactenum}
\end{proposition}

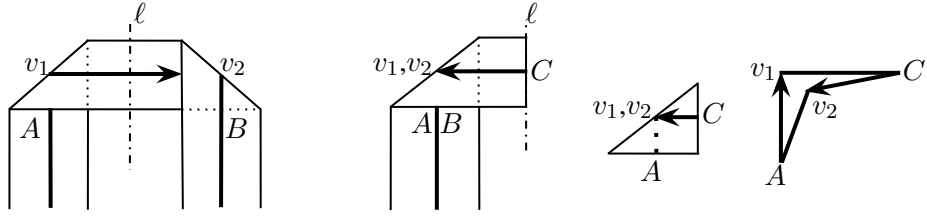
\begin{figure}[htbp]
    \label{fig:linking-1-proof}
    \begin{overpic}{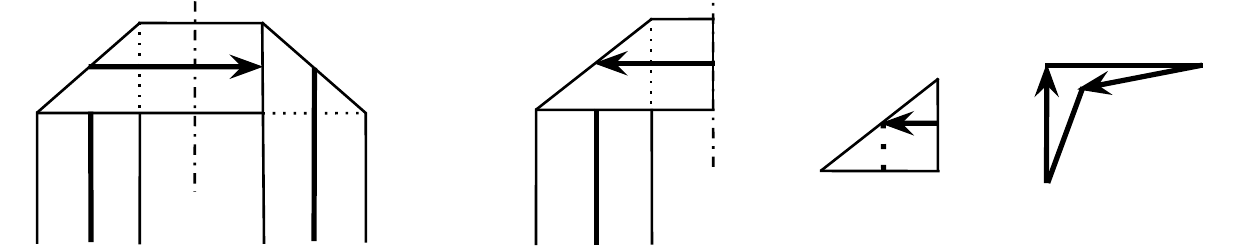}
    \put(4,8){$A$}
    \put(4.5,15){$v_1$}
    \put(16,20){$\ell$}
    \put(25,15){$v_2$}
    \put(25.5,8){$B$}  
    \put(41,15){$v_1$,$v_2$}   
    \put(57.5, 14){$C$}   
    \put(45,8.5){$A$}
    \put(48,8.5){$B$} 
     \put(57,20){$\ell$}
     \put(64,10.5){$v_1$,$v_2$}
     \put(75.5,9.5){$C$}
     \put(69,3.5){$A$}   
     \put(82,3){$A$}
     \put(80.5,14.5){$v_1$}
     \put(96.5,14){$C$}
     \put(87,10.5){$v_2$}
   
     \end{overpic}
    \caption{Construction showing there is a 4-stick unknot diagram whose corresponding folded ribbon unknot has ribbon linking number $+1$ and folded ribbonlength 2.}

\end{figure}

\begin{proof} We will construct such a folded ribbon unknot with $\Lk(\uwf)=+1$.  Take an oriented piece of ribbon and put two right overfolds in it, both with fold angle $\frac{\pi}{2}$, as shown on the left in Figure~\ref{fig:linking-1-proof}. In this figure, the vertices of the unknot diagram $\um$ are denoted by $v_1$ and $v_2$. The point $A$ is where $\um$ enters the first fold, and the point $B$ is where $\um$ leaves the second fold.  We start by folding the piece of ribbon halfway between edge $v_1v_2$ along line $\ell$ as shown second from the left in Figure~\ref{fig:linking-1-proof}. At this point vertex $v_2$ is directly over vertex $v_1$. There is now a fold of fold angle $0$ along $\ell$ and there is a new vertex of  $\um$ denoted by $C$. Next, we join $\um$ together at $A$ and $B$, and relabel this vertex $A$. This creates another fold in the ribbon $\uwf$ with fold angle~$0$.   We can shrink edges $v_1C$ and $Cv_2$ and move $C$ over to the boundary of the folds at vertices $v_1$ and $v_2$. The second from right image in Figure~\ref{fig:linking-1-proof} shows the folded ribbon unknot after this step.  

We have constructed a 4-stick unknot $\um$ with vertices $A, v_1, C, v_2$, and give $\um$ that orientation. An expanded view of $\um$ can be seen in the far right of Figure~\ref{fig:linking-1-proof}, where we have moved vertices $v_1$ and $v_2$ apart. The folded ribbon unknot $\uwf$ is thus a topological annulus. We next compute the ribbon linking number of $\uwf$. The fold at vertex $v_1$ is a right overfold, the fold at vertex $v_2$ is now a left underfold,  and by Remark~\ref{rmk:link-compute} these both contribute $+\frac{1}{2}$ to the ribbon linking number.  Now consider the boundary component of the ribbon parallel to the fold line at $C$. The boundary component in the direction from $A$ to $v_1$ goes under the knot diagram from $C$ to $v_2$ and this crossing has sign $-1$, the boundary component in the direction from $v_2$ to $A$ goes over the knot diagram from $v_1$ to $C$ and this crossing has sign $+1$.  By construction, $\uwf$ has ribbon linking number $\Lk(\uwf)=+1$. Since the folded ribbon unknot is just 2 folds with fold angle $\frac{\pi}{2}$ joined together, we use Corollary~\ref{cor:triangles} to see that the folded ribbonlength is $\Rib(\uwf)=2\times 1=2$. 

If we wish to construct a folded ribbon unknot with $\Lk(\uwf)=-1$, we start by putting two right underfolds in the oriented piece of ribbon, and then follow the rest of the steps.
\end{proof}

This basic construction can be generalized, as can be seen in our next result.

\begin{theorem} \label{thm:linking-n}
There is a folded ribbon unknot $\uwf$ which is a topological annulus, such that 
\begin{compactenum}
\item for any $n\in \N$, the knot diagram $\um$ has $(2n+2)$ sticks,
\item $\uwf$ has ribbon linking number $\Lk(\uwf)=\pm n$, and 
\item $\uwf$ has folded ribbonlength $\Rib(\uwf) = 2n$.
\end{compactenum}
\end{theorem}

\begin{proof}
We will construct the desired folded ribbon unknot with width $w$ and  $\Lk(\uwf)=+n$. We proved the $n=1$ case in Proposition~\ref{prop:linking-1} above.  We divide the proof into two cases: $n$ odd and $n$ even. 

{\bf Case~1:} Assume that $n$ is odd. We start by taking an oriented piece of ribbon, and then add two right overfolds, then two left underfolds, then two right overfolds, and continue in this pattern ending with two right overfolds. There are $n$ pairs of under/overfolds altogether. All the folds have fold angle $\frac{\pi}{2}$, as shown in Figure~\ref{fig:linking-odd}. The case where $n=3$ is illustrated on the left, and the general case for $n$ odd is illustrated on the right.

\begin{figure}[htbp]
     \begin{overpic}[scale=0.65]{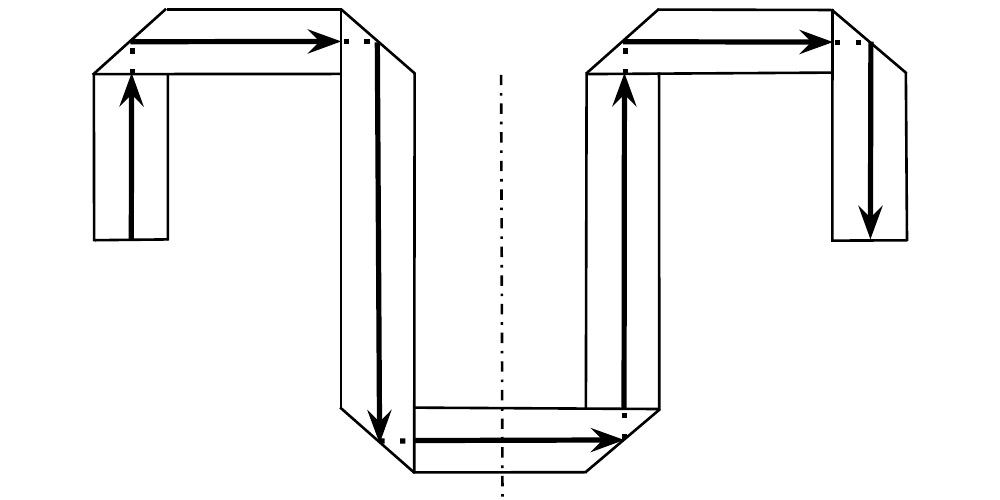}
    \put(10,21){$A$}
    \put(7,46){$v_1$}
    \put(38,46){$v_2$}
    \put(33,3){$v_3$} 
    \put(63,3){$v_4$} 
    \put(46,25){$\ell$} 
    \put(56,46){$v_5$} 
    \put(87,46){$v_6$} 
    \put(85,21){$B$}  
    \put(48, -3){$C$}
     \end{overpic}
      \begin{overpic}[scale=0.65]{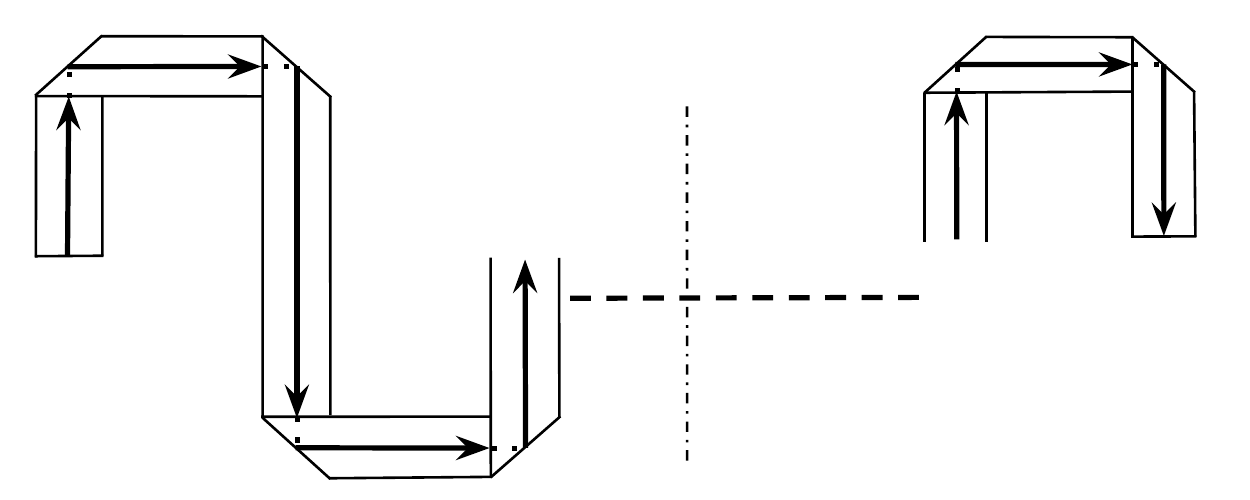}
    \put(3,15.5){$A$}
    \put(1,35.5){$v_1$}
    \put(23.5,35){$v_2$}
    \put(20,2){$v_3$} 
    \put(42.5,2){$v_4$} 
    \put(52,25){$\ell$} 
    \put(65.5,35.5){$v_{2n-1}$} 
    \put(93.5,35.5){$v_{2n}$}  
    \put(92,14.5){$B$}  
    \put(53, 0){$C$}
     \end{overpic}
    \caption{Constructing a folded ribbon unknot which is a topological annulus and whose ribbon linking number is $+n$, where $n$ is odd. On the left, the case $n=3$, and on the right, the general case.}
    \label{fig:linking-odd}
\end{figure}
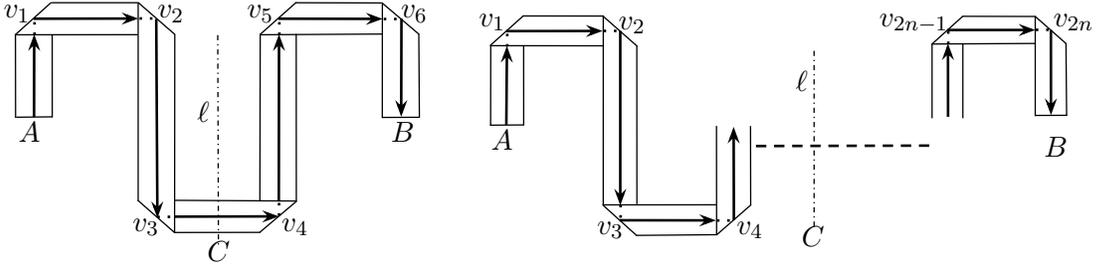

In Figure~\ref{fig:linking-odd}, we see the vertices of the knot diagram $\um$ are labeled $v_1$ through $v_{2n}$. In addition, we assume that the following horizontal and vertical line segments have equal length: 
\begin{align*}
|v_1v_2| & =|v_3v_4| = \dots = |v_{2n-3}v_{2n-2}|=|v_{2n-1}v_{2n}|, \\
|v_2v_3| & = |v_4v_5|=\dots = |v_{2n-4}v_{2n-3}|= |v_{2n-2}v_{2n-1}|, \\
|Av_1| &= |v_{2n}B|.
\end{align*}
  We fold the ribbon in half along the line $\ell$, which intersects the midpoint $C$ of segment $v_nv_{n+1}$, as shown in Figure~\ref{fig:linking-odd}.  After the fold, point $B$ is over point $A$, vertex $v_{2n}$ is over $v_1$, vertex $v_{2n-1}$ is over $v_2$, vertex $v_{2n-2}$ is over $v_3$, etc.  There is now a fold with fold angle $0$ along $\ell$ at vertex $C$.
Next, we join $\um$ together at points $A$ and $B$, and relabel this vertex $A$. This creates another fold with fold angle $0$. We also shrink the piece of ribbon between each pair of folds to have zero length. Since the ribbon has width $w$, after shrinking we see that the edges in $\um$ have length 
\begin{align*} 
\frac{w}{2} &= |Av_1|=|v_nC|=|Cv_{n+1}|=|v_{2n}A|,  \\
w &=|v_1v_2|=|v_2v_3|= \dots =|v_{n-1}v_n| = |v_{n+1}v_{n+2}| =\dots =|v_{2n-1}v_{2n}|. 
\end{align*}

Altogether we have constructed a ($2n+2$)-stick unknot diagram with vertices $$A, v_1, v_2, \dots, v_n, C,v_{n+1}, \dots v_{2n},$$ and give $\um$ that orientation.  Since $\um$ has an even number of edges, the corresponding folded ribbon unknot $\uwf$ is a topological annulus. Since the folds are either right overfolds or left underfolds, we use the same argument as in Proposition~\ref{prop:linking-1} to see that $\uwf$ has ribbon linking number $\Lk(\uwf)=+n$.  Since the folded ribbon unknot $\uwf$ is just $2n$ folds with fold angle $\frac{\pi}{2}$ joined together, we use Corollary~\ref{cor:triangles} to see that the folded ribbonlength is $\Rib(\uwf)=2n\times 1=2n$.

{\bf Case~2:} Assume $n$ is even.  For $n=2$, we take an oriented piece of ribbon, and then add two right overfolds, followed by two left underfolds.  For general $n$ even, we take an oriented piece of ribbon, and then add two right overfolds, then two left underfolds, then two right overfolds, and continue in this pattern ending with two left underfolds. Altogether there are $n$ pairs of over/underfolds. All the folds have fold angle $\frac{\pi}{2}$, as shown on the left in Figure~\ref{fig:linking-2} and in Figure~\ref{fig:linking-even}. In these figures we see the vertices of the knot diagram $\um$ are labeled $v_1$ through $v_{2n}$. In addition, we assume that the following horizontal and vertical line segments have equal length: $|v_1v_2|=|v_3v_4| = \dots =|v_{2n-1}v_{2n}|$, and $|v_2v_3| = |v_4v_5|=\dots = |v_{2n-2}v_{2n-1}|$ and $|Av_1| = |v_{2n}B|$.  

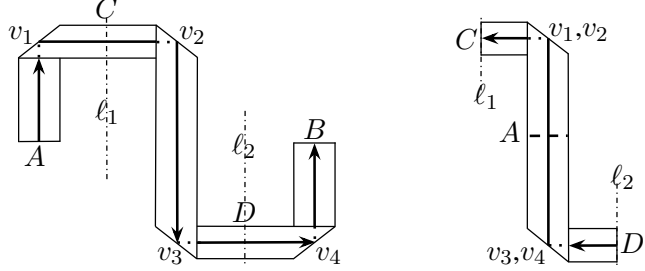
\begin{figure}[htbp]
    \begin{overpic}[scale=0.65]{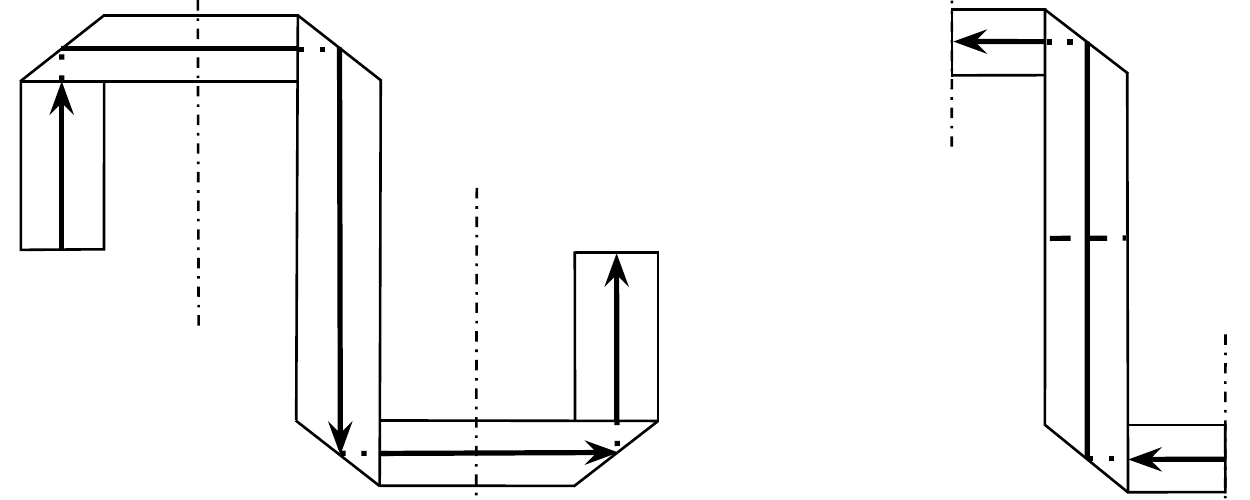}
    \put(2.5,16){$A$}
    \put(0,36.5){$v_1$}
    \put(14,23.5){$\ell_1$}
    \put(14, 40){$C$}
    \put(27.5,36.5){$v_2$}
    \put(24,1){$v_3$} 
    \put(49.5,1){$v_4$} 
    \put(36,18){$\ell_2$}
    \put(36,7){$D$} 
    \put(47.5,20.5){$B$}  
    \put(87,37){$v_1$,$v_2$}
    \put(75,26){$\ell_1$}
    \put(72,35){$C$}
    \put(79,20){$A$}
    \put(77,1){$v_3$,$v_4$}
    \put(97,13){$\ell_2$}
    \put(98.5,2){$D$}
     \end{overpic}
       \caption{Constructing a folded ribbon unknot which is a topological annulus and whose ribbon linking number is $+2$.}
    \label{fig:linking-2}
\end{figure}

Throughout the following construction we will refer to Figures~\ref{fig:linking-2} and~\ref{fig:linking-even}. We start by folding the piece of ribbon corresponding to edge $v_1v_2$ in half along the line $\ell_1$. After the fold, vertex $v_1$ is over $v_2$, and there is now a fold with fold angle $0$ along $\ell_1$ at a new vertex of $\um$ denoted by $C$. The next step is to fold the rest of the ribbon (from edge $v_2v_3$ to edge $v_{2n}B$) in half. In the $n=2$ case, this means folding edge $v_3v_4$ in half along line $\ell_2$. In the general case, this means folding edge $v_{n+1}v_{n+2}$ along the line $\ell_2$.  There is now a fold with fold angle $0$ along $\ell_2$ at a new vertex of $\um$ denoted by $D$. 

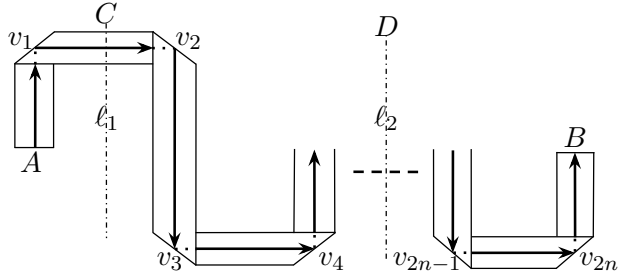
\begin{figure}[htbp]
      \begin{overpic}[scale=0.65]{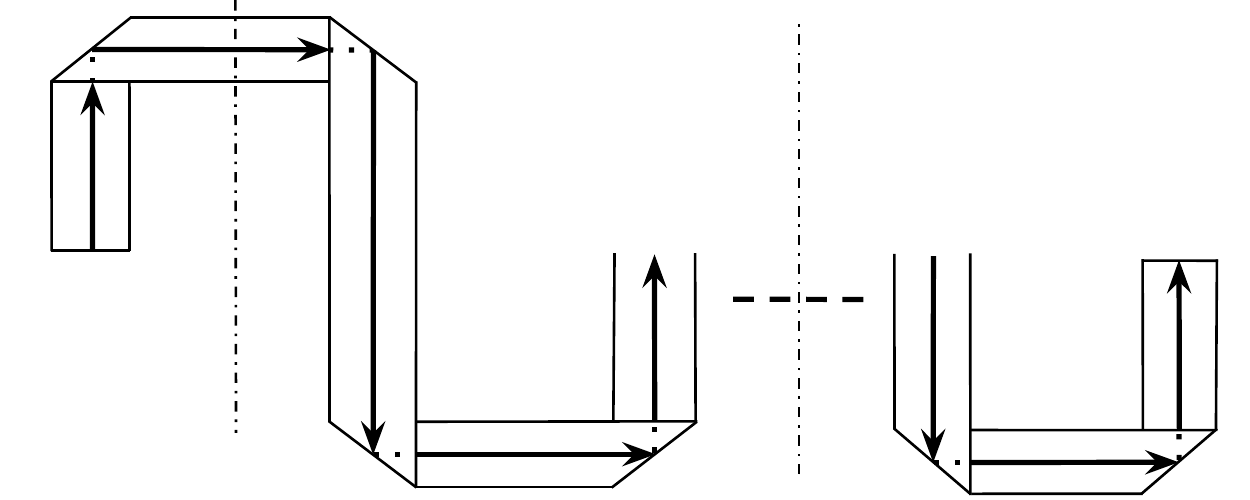}
 \put(5,16){$A$}
    \put(3,36.5){$v_1$}
    \put(17,23.5){$\ell_1$}
    \put(17,40){$C$}
    \put(30,36.5){$v_2$}
    \put(27,1){$v_3$} 
    \put(53,1){$v_4$} 
    \put(62,23.5){$\ell_2$} 
    \put(62, 38){$D$}
    \put(92.5,20){$B$}  
    \put(65,1){$v_{2n-1}$}
    \put(95,1){$v_{2n}$}
     \end{overpic}
       \caption{Constructing a folded ribbon unknot which is a topological annulus and whose ribbon linking number is $+n$, for $n$ even.}
    \label{fig:linking-even}
\end{figure}

At this stage, point $B$ is over point $A$, vertex $v_{2n}$ is over $v_3$, vertex $v_{2n-1}$ is over $v_4$, etc.  Next, we join $\um$ together at points $A$ and $B$, and relabel this point $A$. Since $v_1Av_{2n}$ is a straight line, we replace the two edges $v_1A$ and $Av_{2n}$ with one edge $v_1v_{2n}$.  This has been illustrated for the $n=2$ case on the right in Figure~\ref{fig:linking-2}.
We again shrink the piece of ribbon between each pair of folds to have zero length. Since the ribbon has width $w$, after shrinking we see that the edges in $\um$ have length \begin{align*}
\frac{w}{2} &= |v_1C| = |Cv_2| =|v_{n+1}D|=|Dv_{n+2}|, \\ 
w&=|v_2v_3|=|v_3v_4|= \dots =|v_nv_{n+1}| = |v_{n+2}v_{n+3}| =\dots =|v_{2n-1}v_{2n}|=|v_{2n}v_1|.
\end{align*}

 Altogether we have constructed a ($2n+2$)-stick unknot diagram $\um$ with vertices $$v_1, C, v_2, \dots, v_{n+1}, D,v_{n+2}, \dots v_{2n},$$ and give $\um$ that orientation. Since $\um$ has an even number of edges, the corresponding folded ribbon unknot $\um$ is a topological annulus.  The argument now follows in the exact same way to the $n$ odd case. Since the folds are either right overfolds or left underfolds, then $\uwf$ has ribbon linking number $\Lk(\uwf)=+n$. Since the folded ribbon unknot $\uwf$ is just $2n$ folds with fold angle $\frac{\pi}{2}$ joined together, we use Corollary~\ref{cor:triangles} to find the folded ribbonlength is $\Rib(\uwf)=2n\times 1=2n$. 

If we wish to construct a folded ribbon unknot with $\Lk(\uwf)=-n$, we follow the same steps in Cases 1 and 2, but switch all overfolds to underfolds and vice versa.
\end{proof}

\begin{corollary}  \label{cor:n-unknot2}
For any  $n\in \N$, the minimum folded ribbonlength of any folded ribbon unknot $\uwf$ which is a topological annulus with ribbon linking number $\Lk(\uwf)=\pm n$ is bounded above by $\Rib([\uwf])\leq 2n$.
\end{corollary}

Recall that in Theorem~\ref{thm:n-unknot1}, we gave different upper bounds on the minimum folded ribbonlength for folded ribbon unknots of both topological types which depended on the ribbon linking number.  It turns out that when $\uwf$ is a topological annulus, the best upper bound is given by Corollary~\ref{cor:n-unknot2}.  To see this, we simply compare values.
\begin{compactitem}
\item When $\Lk(\uwf)=\pm 1$,  we have $\Rib([\uwf])\leq 2<4$. 
\item When $\Lk(\uwf)=\pm n$ for all other $n\in \N$, we have $\Rib([\uwf])\leq 2n\leq 2n\cot(\frac{\pi}{2n})$. (Note there is equality if and only if $n=2$.)
\end{compactitem}

The construction in Theorem~\ref{thm:linking-n} also allows us to prove the following theorem.

\begin{theorem} \label{thm:writhe0}
The minimum folded ribbonlength of any folded ribbon unknot $\uwf$ which is a topological annulus with ribbon linking number $\Lk(\uwf)=\pm n$ and writhe $\Wr(\uwf) =0$ is
$\Rib([\uwf])=2n$.
\end{theorem}

\begin{proof} When $n=0$, the result is proven by the $2$-stick unknot.
Now consider $n\in\N$.
Theorem~\ref{thm:RibLowerBd} tells us that the folded ribbonlength of such a folded ribbon unknot is bounded from below by $2n$. We give an upper bound  by considering the folded ribbon unknot constructed in Theorem~\ref{thm:linking-n}. This has ribbon linking number $\pm n$, and has folded ribbonlength $2n$. What remains is to show that this folded ribbon unknot has writhe 0. 

Since the construction in Theorem~\ref{thm:linking-n} gives unknots with non-regular diagrams, we compute the writhe using the integral formula, using the same argument given for the $2$-stick unknot in Section~\ref{sect:LkTwWr}. We will refer to the notation given in Theorem~\ref{thm:linking-n}. Assume that $n$  is odd, and suppose the unknot diagram $\um$ has vertices $A$ and $C$ at height 0, vertices $v_1, v_2, \dots, v_n$, at height $-1$, and vertices $v_{n+1},\dots, v_{2n}$ at height $+1$. Then the writhe is zero for all directions except the vertical. Now consider the $n$ even case. Here, vertices $C$ and $D$ are at height $0$, vertices $v_2, v_3, \dots v_{n+1}$ are at height $-1$, and vertices $v_{n+2}, \dots, v_{2n},v_1$ are at height $+1$. Using the same reasoning, we see the writhe is zero.
\end{proof}

It is an interesting question to wonder if there are folded ribbon unknots with nonzero writhe that have smaller folded ribbonlength than Theorem~\ref{thm:writhe0}. We suspect not. Recall from Remark~\ref{rmk:link-compute} that crossings contribute $\pm 1$ to the ribbon linking number. Also note that Proposition~\ref{prop:extended-fold} can be applied to crossings to show they contribute at least $2$ to the folded ribbonlength. This is no better (or worse) than folds and leads us to conjecture.

\begin{conjecture} \label{conj:link-n}
The minimum folded ribbonlength of any folded ribbon unknot $\uwf$ which is a topological annulus with ribbon linking number $\Lk(\uwf)=\pm n$ is $\Rib([\uwf])=2n$.
\end{conjecture}

\subsection{Knots and ribbon linking number}
It is natural to ask how the folded ribbonlength of non-trivial knots is affected by different ribbon linking numbers. We can easily change the ribbon linking number of a folded ribbon knot by adding pairs of folds (with appropriate sign) as in Lemma~\ref{lem:two-folds} and Theorem~\ref{thm:linking-n}. We can then use Theorem~\ref{thm:linking-n} to give an upper bound on the folded ribbonlength of the added folds.

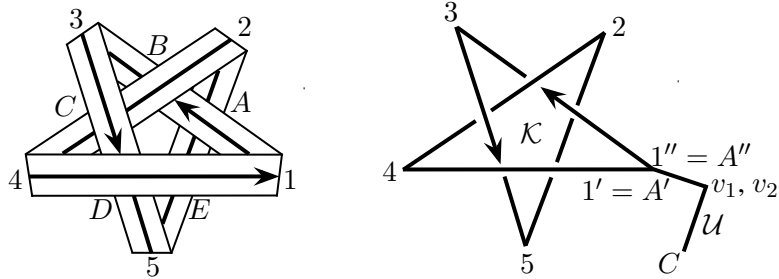
\begin{figure}[htbp]
      \begin{overpic}{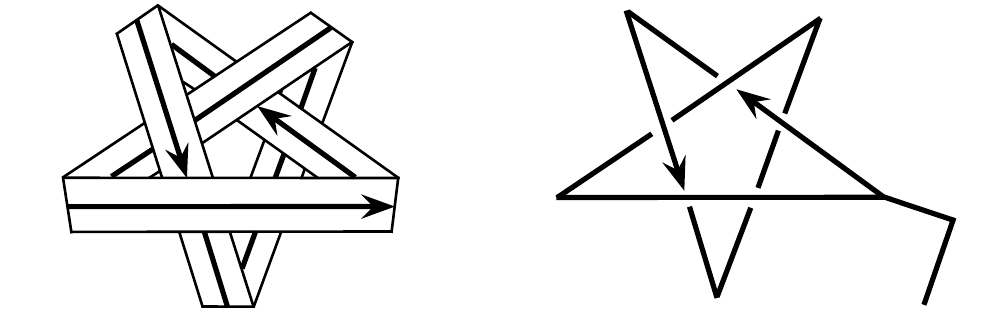}
       \put(40,9){$1$}
       \put(34,29){$2$}
       \put(12,30){$3$}
       \put(4,9){$4$}
       \put(22,-2.5){$5$}
       \put(33, 19){$A$}
       \put(22,26.5){$B$}
       \put(10,18.5){$C$}
       \put(14.5,5){$D$}
       \put(27.5,5){$E$}
       \put(79,7.5){$1'=A'$}
       \put(88,12){$1''=A''$}
       \put(83,29){$2$}
       \put(61,31){$3$}
       \put(53,10){$4$}
       \put(71,-2){$5$}
       \put(96,8.5){$v_1$, $v_2$}
       \put(89,-2){$C$}
       \put(71, 15){$\mk$}
       \put(95, 3){$\um$}
      \end{overpic}
     \caption{On the left, the folded ribbon trefoil knot has ribbon linking number $-7$. On the right, creating a connected sum of the trefoil  $\mk$ with a 4-stick unknot $\um$. The corresponding folded ribbon trefoil knot has linking number $-6$ or $-8$, depending on the choice of folding information at vertices $v_1$ and $v_2$.}
  \label{fig:trefoil-linking}
\end{figure}

As a simple example of this, consider the folded ribbon trefoil knot $\kwf$ shown on the left in Figure~\ref{fig:trefoil-linking}. The folds $1$ and $5$ are left underfolds and have sign $+1$, while folds $2$, $3$, and $4$ are left overfolds and have sign $-1$.  Crossings $A$, $B$, $C$, and $D$ all have sign $-1$, and crossing $E$ has sign $+1$. The folded ribbon trefoil knot is a M\"obius band, and using Remark~\ref{rmk:link-compute} we see that the ribbon linking number $\Lk(\kwf) = +2-3+2(-4+1)=-7$. From \cite{Kauf05, KMRT} we know that the folded ribbonlength of this trefoil when in its ``tight'' pentagonal form (as in Figure~\ref{fig:trefoil-pentagon} right) is $\Rib(\kwf)=5\cot(\frac{\pi}{5})\approx 6.882$. We can change the ribbon linking number of $\kwf$ by creating a new folded ribbon trefoil knot $\kpwf$ which is the connected sum of $\kwf$ with a folded ribbon unknot $\uwf$ with ribbon linking number $n$ as constructed as in Theorem~\ref{thm:linking-n}. To do this, we can break the trefoil knot $\mk$ at one of its vertices, then join the unknot $\um$. As a specific example, on the right in Figure~\ref{fig:trefoil-linking}, we have broken the trefoil knot $\mk$ at vertex $1$. We have also broken a 4-stick unknot $\um$ with vertices $A,v_1,C, v_2$ (from Proposition~\ref{prop:linking-1}) at vertex $A$, then created the connected sum $\mk'=\mk+\um$.  As seen in Proposition ~\ref{prop:linking-1}, the corresponding folded ribbon unknot $\uwf$ has ribbon linking number $\pm 1$ depending on the folding information at vertices $v_1$ and $v_2$. Since $\um$ has an even number of sticks, the connected sum $\kpwf$ is still a M\"obius band. We can carefully arrange the folds and crossings at vertices $1'=A'$ and $1''=A''$ of  the new folded ribbon knot $\kpwf$ so that $\kpwf$ has
\begin{compactitem}
\item ribbon linking number $\Lk(\kpwf)=-6$ or $-8$, and 
\item folded ribbonlength $\Rib(\kpwf)\approx 5\cot(\frac{\pi}{5}) + 2$. (We may need to extended edges $A'v_1$ and $v_2A''$ to account for the new folds at vertices $1'=A'$ and $1''=A''$.)
\end{compactitem}
Since the connected sum has $5+4=9$ edges, it is entirely possible that a different configuration has smaller folded ribbonlength.  This example motivates the following theorem.

\begin{theorem}\label{thm:knot+n}
Suppose a knot has a polygonal knot diagram $\mk$, and assume the corresponding folded ribbon knot $\kwf$ has ribbon linking number $\Lk(\kwf)=m$ and folded ribbonlength $\Rib(\kwf)=M$. Then there is a polygonal knot diagram $\mk'$ such that $\mk$ and $\mk'$ have the same knot type, and for any $n\in\Z$ the corresponding folded ribbon knot $\kpwf$ has 
\begin{compactenum}
\item the same topological type as $\kwf$,
\item ribbon linking number $\Lk(\kpwf)=m+n$.
\end{compactenum}
The minimum folded ribbonlength is bounded from above by $\Rib([\kpwf])\leq M+2|n| + L$, where
$$ L=\begin{cases} 2\cot(\frac{\alpha}{4}) - \cot(\frac{\alpha}{2}) & \text{when $\alpha$ is acute},\\ 2\cot(\frac{\alpha}{4}) - \cot(\frac{\pi}{2}-\frac{\alpha}{2}) & \text{when $\alpha$ is obtuse}.
\end{cases}
$$
Here, $\alpha$ is the fold angle of the vertex of $\mk$ used to create $\mk'$.
\end{theorem}

\begin{proof}  Theorem~\ref{thm:linking-n} says there is a $(2n+2)$-stick unknot diagram $\um$ whose corresponding folded ribbon unknot $\uwf$ has ribbon linking number $\Lk(\uwf)=n$ and folded ribbonlength $\Rib(\uwf)=2|n|$. In this proof, we will carefully construct the connected sum $\mk'=\mk\#\um$ so that the corresponding folded ribbon knot $\kpwf$ has the desired properties. We first observe that since $\um$ is an unknot, then $\mk'$ is of the same knot type as $\mk$. We next assume that $\kwf$ and $\uwf$ have been appropriately scaled so they have the same ribbon width. 

To get started with our connected sum, we choose a vertex $w_k$ of $\mk$ which is in the convex hull of the knot diagram $\mk$. We cut the knot diagram at $w_k$ creating two points $w_k'$ and $w_k''$. We assume that the orientation of the (broken) knot diagram is $\dots w_{k-1}, w_k', w_k'', w_{k+1}\dots$.  On the right in Figure~\ref{fig:connect-sum-over}, we see vertex $w_k$ in the knot diagram for $\mk$. The dotted line $\ell$ indicates the position of the fold line of the corresponding folded ribbon $\uwf$, and without loss of generality, we have assumed $\mk$ turns to the left at $w_k$.

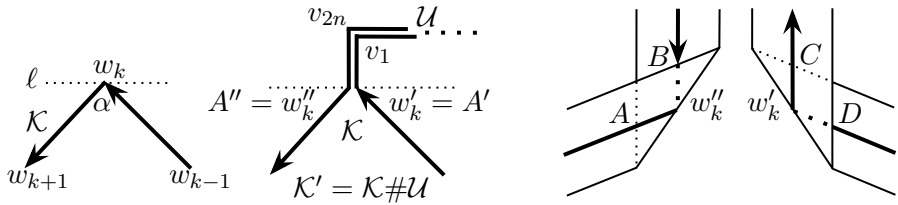
\begin{figure}[htbp]
      \begin{overpic}{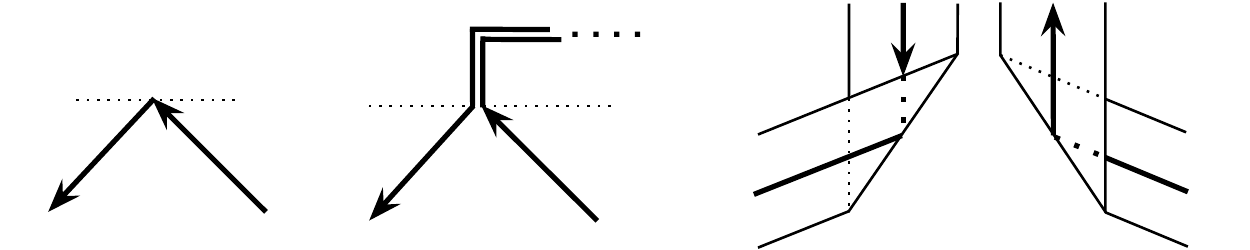}
       \put(11,13){$w_k$}
       \put(2,1.5){$w_{k+1}$}
       \put(19,1.5){$w_{k-1}$}
       \put(11, 9){$\alpha$}
       \put(4,7){$\mk$}
       \put(4,11.5){$\ell$}
       \put(42,9){$w_k'=A'$}
       \put(23,9){$A''=w_k''$}
       \put(39.5,14.5){$v_1$}
       \put(33.5,18.5){$v_{2n}$}
       \put(37,6){$\mk$}
       \put(45,18){$\um$}
       \put(32,0){$\mk'=\mk\#\um$}
       \put(79.5,9){$w_k'$}
       \put(73.5,9){$w_k''$}
       \put(65,7.75){$A$}
       \put(69,13.75){$B$}
       \put(85,14){$C$}
       \put(89,7.75){$D$}
   
       \end{overpic}
       \caption{ On the left and center, we see the detail of the join in the connected sum between $\mk$ and $\um$ from Theorem~\ref{thm:knot+n}. On the right, we see the new folds at vertices $w_k'$ and $w_k''$ as discussed in Case~1(a).}
    \label{fig:connect-sum-over}
\end{figure}

{\bf Case~1:} Assume $n$ is odd. Take the unknot diagram $\um$ with $(2n+2)$ edges constructed in Case~1 of the proof of Theorem~\ref{thm:linking-n}. Then cut $\um$ at vertex $A$, creating two vertices labeled $A'$ and $A''$. The orientation of the (broken) unknot is $A',v_1, \dots,v_n, C, v_{n+1},\dots v_{2n},  A''$.  

To create the connected sum of $\mk$ and $\um$, we join $w_k'$ to $A'$ and $w_k''$ to $A''$ and relabel the vertices $w_k'$ and $w_k''$. This creates the knot $\mk'=\mk\#\um$. The join between $\um$ and $\mk$ is shown in the middle in Figure~\ref{fig:connect-sum-over}, where we have moved vertices $v_1$ and $v_{2n}$ apart for clarity. When we make the connected sum, we arrange the line segments $A'v_1$ and $v_{2n}A''$ to be perpendicular to the fold line at $w_k$.  By construction,  the unknot diagram lies on one side of the line $\ell$,  and so $\um$ lies in the exterior of the convex hull of $\mk$.   We assumed that $\mk$ turned to the left at $w_k$. This means as we traverse $\mk'$, we turn to the right at $w_k'$ and at $w_k''$. If necessary, we can extend the length of edges containing $w_k'$ and $w_k''$ to allow room for the new folds at $w_k'$ and $w_k''$. We discuss these details at the end of the proof.

{\em How does the number of edges change?} If the knot diagram $\mk$ has $p$ edges, then $\mk'=\mk\#\um$ has $p+(2n+2)$ edges. This means the corresponding folded ribbon knot $\kpwf$ is the same topological type as $\kwf$. 

{\em How does the ribbon linking number change?}  By construction, the only places that $\uwf$ and $\kwf$ have changed or interact with one another are are in a neighborhood of $w_k'$ and $w_k''$. Since the fold at vertex $A$ of $\uwf$ has fold angle $0$, this fold does not contribute to the ribbon linking number of $\uwf$.  
However, the fold at vertex $w_k$ does contribute to the ribbon linking number of $\kwf$.  We will arrange the folds and crossings at $w_k'$ and $w_k''$ so their contributions to the ribbon linking number match. There are two sub-cases to consider here.

{\bf Case~1(a)} Assume that in $\kwf$, there is a left underfold at $w_k$ which has sign $+1$.    When making the connected sum, we replace the left underfold at $w_k$ with two right overfolds at $w_k'$ and $w_k''$, and both of these have sign $+1$. This is illustrated on the right in Figure~\ref{fig:connect-sum-over}, where we have separated the folded ribbons at $w_k'$ and $w_k''$ to make it easier to see what is going on. When $w_k$ has a left underfold, we choose that vertex $w_k''$ is over $w_k'$, and that the unknot follows (so vertex $v_{2n}$ is over vertex $v_1$, etc). In this set up, there are two other places that where the knot diagram $\mk'$ interacts with the boundary of the folded ribbon $\kpwf$. The first is where the knot diagram at point $C$ crosses under the ribbon boundary at point $B$. The second is where the ribbon boundary at point $C$ crosses under the knot diagram at point $B$.  Both of these crossings have sign $-1$.  We note that the piece of folded ribbon at point $A$ does not interact with the piece of folded ribbon at point $D$.

When $\kwf$ (and hence $\kpwf$) is a topological annulus, then the folded ribbon has two boundary components.  The two right overfolds at $w_k'$ and $w_k''$ each contribute $+\frac{1}{2}$ to $\Lk(\kpwf)$. By following one boundary component, we see just one of the crossings at the points $B$ and $C$ between the knot diagram and the ribbon's boundary contribute $-\frac{1}{2}$ to $\Lk(\kpwf)$. Altogether, this matches the contribution to the ribbon linking number of $+\frac{1}{2}$ to $\Lk(\kwf)$ by the left underfold at vertex $w_k$.

When $\kwf$ (and hence $\kpwf$) is a M\"obius strip, then the folded ribbon has just one boundary component.  The two right overfolds at $w_k'$ and $w_k''$ each contribute  $+1$ to $\Lk(\kpwf)$, and both of the crossings at the points $B$ and $C$ between the knot diagram and the ribbon's boundary contribute $-\frac{1}{2}$ to $\Lk(\kpwf)$. Altogether, this matches the contribution of $+1$ to $\Lk(\kwf)$ by the left underfold at vertex $w_k$.

 {\bf Case~1(b)} Assume that in $\kwf$, there is a left overfold at $w_k$ which has sign $-1$. When making the connected sum, we replace the left overfold at $w_k$ with two right underfolds at $w_k'$ and $w_k''$, which both have sign $-1$. Next,  we assume that $w_k'$ is over $w_k''$, and this carries over to the unknot $\um$.  This means $v_1$ is over $v_{2n}$, etc. As in Case~1(a), the knot diagram $\mk'$ interacts with the boundary of the folded ribbon $\kpwf$ at the equivalent of points $B$ and $C$ on the right of Figure~\ref{fig:connect-sum-over}. Our assumption that $w_k'$ is over $w_k''$ means that these two crossings have sign $+1$. We follow the same argument as in Case~1(a) to deduce that the local contributions to the ribbon linking number are identical.
In both cases 1(a) and 1(b), we see that
$$\Lk(\kpwf)=\Lk(\kwf)+\Lk(\uwf) = m+n.$$

{\bf Case~2:} Assume $n$ is even. We construct the connected sum between $\mk$ and $\um$ in almost the same way as in Case~1. Take the unknot diagram $\um$ with $(2n+2)$ edges constructed in Case~2 of the proof of Theorem~\ref{thm:linking-n}. In this case, we cut $\um$ at vertex $C$ creating two vertices labeled $C'$ and $C''$. We assume that order of vertices as we travel along the (broken) unknot is $C', v_2,\dots,v_{n+1}, D, v_{n+2}, \dots, v_{2n}, v_1, C''$. To create the connected sum of $\mk$ and $\um$, we join $C'$ to $w_k'$, and $C''$ to $w_k''$.   The argument then follows in an identical fashion to the Case~1, with two sub-cases depending on whether there is a left underfold or left overfold at vertex $w_k$.

{\em How does the ribbonlength change?} When making the connected sum, the ribbonlength of $\uwf$ and $\kwf$ only increases if we need to extend the edges adjacent to vertices $w_k'$ and $w_k''$ to create room for the new folds.  What do we mean? Using the language and formulas from Proposition~\ref{prop:extended-fold}, we replace the extended fold at $w_k$ with two new extended folds at $w_k'$ and $w_k''$.  Suppose the fold angle at $w_k$ is $\alpha$, then the new fold angles at $w_k'$ and $w_k''$ are both $\theta=\pi-\frac{\alpha}{2}$. Since $0<\alpha< \pi$, then $\frac{\pi}{2}<\theta<\pi$. It's also useful to note that $\frac{\pi}{2}-\frac{\theta}{2} = \frac{\pi}{2}-\frac{1}{2}(\pi-\frac{\alpha}{2}) = \frac{\alpha}{4}$.

 If $\alpha$ is acute, then the extended fold at $w_k$ has folded ribbonlength $\cot(\frac{\alpha}{2})$. 
 If $\alpha$ is obtuse, then the extended fold at $w_k$ has folded ribbonlength $\cot(\frac{\pi}{2}-\frac{\alpha}{2})$.
In either case, the new extended folds at $w_k'$ and $w_k''$  have folded ribbonlength $2\cot(\frac{\pi}{2}-\frac{\theta}{2})= 2\cot(\frac{\alpha}{4})$.
We let $L$ represent the change in ribbonlength, and we obtain
$$ L=\begin{cases} 2\cot(\frac{\alpha}{4}) - \cot(\frac{\alpha}{2}) & \text{when $\alpha$ is acute},\\ 2\cot(\frac{\alpha}{4}) - \cot(\frac{\pi}{2}-\frac{\alpha}{2}) & \text{when $\alpha$ is obtuse}.
\end{cases}
$$
This means that 
$$\Rib(\kpwf)\leq \Rib(\uwf)+\Rib(\kwf)+L\leq M+2|n| + L.$$
\end{proof}

\begin{remark} 
We do not expect that the upper bound in Theorem~\ref{thm:knot+n} to be sharp. Since the number of edges in the new knot $\mk'$ is greater than the number of edges in $\mk$, we expect the flexibility given by the extra edges means that the folded ribbonlength of $\kpwf$ will be much smaller than that of $\kwf$. 

The simplest example of this phenomenon is found in the work of Kennedy {\em et al.} \cite{KMRT}. They give two examples of a $(5,2)$ torus knot. One is made with 5 edges, and the other is made with 7 edges, and the latter has smaller folded ribbonlength than the former. In \cite{DKTZ}, we show that the $(5,2)$ torus knot made with 7 edges differs to the one made with 5 edges by one full twist. This means the ribbon linking number differs by $\pm 1$.  

As another example, consider $(2,p)$ torus knots. Kennedy {\em et al.} show that for $p\geq 7$ there are $(2,p)$ torus knots with $p$ edges and with folded ribbonlength $p\cot(\frac{p}{2})$. In \cite{Den-FRF}, we give two different constructions. The first uses $2p+2$ edges to construct a $(2,p)$ torus knot, the second uses $4p$ edges to construct a $(p,2)$ torus knot. However, both constructions bound the folded ribbonlength above by $2p$, which is much smaller than $p\cot(\frac{p}{2})$. We did not try and compute the ribbon linking number of any of these constructions of folded ribbon $(2,p)$ torus knots. Instead, these examples show that adding in extra edges allowed us to greatly decrease the folded ribbonlength.
\end{remark}

\subsection{Summary}
In this paper we have looked at one problem: finding upper and lower bounds on the minimum folded ribbonlength of folded ribbon unknots with respect to ribbon equivalence. We have given upper bounds on the minimum folded ribbonlength for folded ribbon unknots which are topological annuli and M\"obius bands for any ribbon linking number. We proved that the minimum folded ribbonlength for $3$-stick unknots occurs when the unknot is an equilateral triangle. We generalized this proof and showed that for folded ribbon unknots whose diagrams have obtuse fold angles, the minimum folded ribbonlength occurs when the unknot is a regular $n$-gon. Finally, we proved the minimum folded ribbonlength is $2n$ for folded ribbon unknots which are topological annuli, have writhe 0, and with ribbon linking number $\pm n$.  This work is far from complete. Along the way we made many conjectures. For example, what about folded ribbon unknots with nonzero writhe? Finally, very little, if anything, is known about the folded ribbonlength of non-trivial folded ribbon knots with any given ribbon linking number.


\section{Acknowledgments}
We wish to thank Jason Cantarella for many helpful conversations about folded ribbon knots.

Denne's research has been funded over many summers by Lenfest Grants from Washington \& Lee University, most recently in 2018, 2019, 2020, and 2022. Larsen's research was funded by Washington \& Lee Summer Research Scholars program in 2020.

The work on folded ribbon knots has been developed with Denne's undergraduate students over many years. We wish to thank the students who have contributed to our understanding of folded ribbon knots.
\begin{compactitem}
\item Shivani Aryal and Shorena Kalandarishvili: funded by Smith College's 2009 Summer Undergraduate Research Fellowship program.
\item Eleanor Conley, Emily Meehan and Rebecca Terry: funded by the Center for Women in Mathematics at Smith College Fall 2011, which was funded by NSF grant DMS 0611020. 
\item Mary Kamp and Catherine (Xichen) Zhu: funded by 2015 Washington \& Lee Summer Research Scholars program.
\item Corinne Joireman and Allison Young: funded by 2018 Washington \& Lee Summer Research Scholars program.
\item John Carr Haden and Troy Larsen: funded by 2020 Washington \& Lee Summer Research Scholars program.
\end{compactitem}


\bibliography{folded-ribbons}{}
\bibliographystyle{plain}

\appendix
\section{Detail of Proposition~\ref{prop:n-linking} }\label{append:n-linking}
In this section, we summarize the details in the proof of Proposition~\ref{prop:n-linking} (which proved that ribbon linking number of the corresponding folded ribbon unknot $\uwf$ is a complete invariant). Here, we assume that $\um$ be a non-degenerate convex $n$-gon for $n\geq 3$.  Proposition~\ref{prop:n-linking} proved that the only possible ribbon linking numbers are as follows:
\begin{compactenum}
\item when $n\geq 3$ and is odd,  then $\Lk(\uwf) = \pm 1, \pm 3, \dots , \pm n$;
\item when $n\geq 4$ and is even, then $\Lk(\uwf) = 0, \pm 1, \pm 2, \dots , \pm\frac{n}{2}$.
\end{compactenum}
Moreover, up to permutation among the vertices, the folding information determines the ribbon linking number and the ribbon linking number determines the folding information.  We now state precisely what this last sentence means. In both of the following cases, we assume that $\um$ is oriented in the counterclockwise direction.

Assume that $n\geq 3$ is odd. Then
\begin{itemize}
\item  $Lk(\uwf)= + n$ if and only if there are $n$ underfolds,
\item    $Lk(\uwf)= + (n-2)$ if and only if there are $n-1$ underfolds and $1$ overfold,
 \\
 \vdots

\item  $Lk(\uwf)= +1$ if and only if there are $\lceil\frac{n}{2}\rceil$ 
underfolds and $\lfloor\frac{n}{2}\rfloor$ overfolds,
\item  $Lk(\uwf)= -1$ if and only if there are $\lfloor\frac{n}{2}\rfloor$ underfolds and $\lceil\frac{n}{2}\rceil$ 
overfolds,
\\
\vdots
\item  $Lk(\uwf)= - (n-2)$ if and only if there are  $n-1$ overfolds and $1$ underfold,
\item $Lk(\uwf)= - n$ if and only if there are $n$ overfolds. 
\end{itemize}

Assume that $n\geq 4$ is even. Then
\begin{itemize}
\item  $Lk(\uwf)= + \frac{n}{2}$ if and only if there are $n$ underfolds,
\item    $Lk(\uwf)= + (\frac{n}{2}-1)$ if and only if there are $n-1$ underfolds and 1 overfold,
 \\
 \vdots

\item  $Lk(\uwf)= +1$ if and only if there are $\frac{n}{2}+1$ underfolds and $\frac{n}{2}-1$ overfolds,
\item $Lk(\uwf)= 0$ if and only if there are $\frac{n}{2}$ underfolds and $\frac{n}{2}$ overfolds,
\item  $Lk(\uwf)= -1$ if and only if there are $\frac{n}{2}-1$ underfolds and $\frac{n}{2}+1$ overfolds,\\
\vdots
\item  $Lk(\uwf)= - (\frac{n}{2}-1)$ if and only if there are $n-1$ overfolds and $1$ underfold,
\item $Lk(\uwf)= - \frac{n}{2}$ if and only if there are $n$ overfolds. 
\end{itemize}


\section{Proof of Proposition~\ref{prop:equiangular}}\label{append:equiangular}

In this section we prove Proposition~\ref{prop:equiangular}:
Suppose $\um$ is a non-degenerate $n$-gon for $n\geq 4$ whose  interior angles satisfy $\frac{\pi}{2}\leq \alpha_i<\pi$. Then the folded ribbon unknot $\uwf$ has ribbon width bounded from above by $w\leq \frac{1}{n\cot(\frac{\pi}{n})}$, and $w$ achieves its maximum when $\um$ is equiangular.

Just as with Theorem~\ref{thm:equilateral-3-stick}, we need to set up notation to properly describe $\um$ and its corresponding folded ribbon knot $\uw$. Referring to Figure~\ref{fig:obtuse-angles1}, we let  $\um$ have vertices denoted by $v_1, v_2, \dots, v_n$, and corresponding interior angles (the fold angles) denoted by $\alpha_1,\alpha_2,\dots, \alpha_n$. For a ribbon of width $w$, denote the fold lines at each vertex  $v_i$ by $v_i'v_i''$. Without loss of generality, we encounter $v_1', v_1, v_1''$; $v_2', v_2, v_2''$; \dots; $v_n', v_n, v_n''$ when we traverse $\um$ in a counterclockwise direction. As before, we extend the fold lines until they intersect, labeling the point of intersection of the extended fold lines at $v_i$ and $v_{i+1}$ by $w_{i,i+1}$. The exterior boundary of the folded ribbon consists of lines $v_1''v_2', v_2''v_3', \dots, v_n''v_1'$, and we denote the lengths of these lines by  $x_1=|v_1''v_2'|,$ $x_2=|v_2''v_3'|$, \dots, $x_n=|v_n''v_1'|$.   
 As we proved in Lemma~\ref{lem:fold-meet},  the maximum ribbon width occurs when the fold lines intersect in the exterior region of the knot diagram. That is, when one of $x_1,\dots x_n$ is $0$. We can use this to work out which angles $\alpha_1,\dots,\alpha_n$ give the maximum width.

\begin{figure}[htbp]
    \begin{overpic}{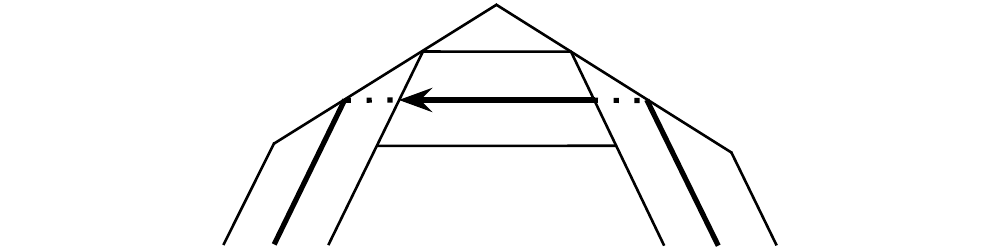}
     \put(73,10){$v_1'$}
     \put(65,15){$v_1$}
     \put(57.5,20){$v_1''$}
     \put(61,12.5){$\alpha_1$}
     \put(49,17){$x_1$}
     \put(48, 25){$w_{1,2}$}
     \put(23,10){$v_2''$}
     \put(31,16){$v_2$}
     \put(38.5,21){$v_2'$}
     \put(34.5,12.5){$\alpha_2$}
     \put(30.5, 4){$x_2$}
     \put(65.5,4){$x_n$}
    \end{overpic}
    \caption{Part of a folded ribbon $n$-stick unknot showing the labels of the vertices, angles, fold lines, and lengths.}
    \label{fig:obtuse-angles1}
\end{figure}
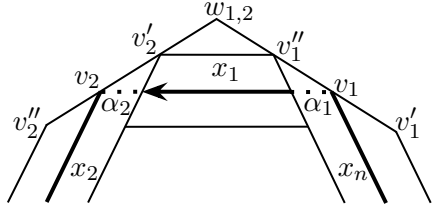

\begin{proof} We start by using the same geometric argument from Theorem~\ref{thm:equilateral-3-stick} to see that 
$$ x_i =|v_iv_{i+1}| - \frac{w}{2} \tan\frac{\alpha_i}{2}-\frac{w}{2} \tan\frac{\alpha_{i+1}}{2} < |v_iv_{i+1}|, $$ where the indices are taken mod $n$
We also assume that the length of $\um$ is fixed and that $|v_1v_2| + |v_2v_3|+ \dots +|v_nv_1| =1$. This allows us to deduce that $0\leq x_1+\dots +x_n <1$.  As in Theorem~\ref{thm:equilateral-3-stick}, we substitute and rearrange, to see that
\begin{equation}
   0\le \sum_{i=1}^n \tan(\frac{\alpha_i}{2})\le\frac{1}{w}. 
 \label{eq:n-inequality}
\end{equation}

Thus in order to maximize $w$, we need to minimize $\sum_{i=1}^n \tan(\frac{\alpha_i}{2})$ remembering that 
the sum of the interior angles of an $n$-gon is given by $\pi(n-2)$. This means we use the method of Lagrange multipliers to minimize the function $f(\alpha_1,\dots, \alpha_n)= \sum_{i=1}^n \tan(\frac{\alpha_i}{2}) $ subject to the constraint function $g(\alpha_1,\dots, \alpha_n)=\alpha_1+\alpha_2+\dots +\alpha_n- \pi(n-2)=0 $. 
We first set the gradient Lagrangian function $\mathcal{L}(\alpha_1,\dots,\alpha_n, \lambda)=f(\alpha_1,\dots,\alpha_n)-\lambda g(\alpha_1,\dots, \alpha_n$ equal to $0$. That is,
\begin{align*}
    0&=\nabla \sum_{i=1}^n \tan(\frac{\alpha_i}{2}) - \lambda\nabla(-\pi(n-2) + \sum_{i=1}^n \alpha_i ),\\
    0&=\langle\frac{1}{2}\sec^2(\frac{\alpha_1}{2}),\dots, \frac{1}{2}\sec^2(\frac{\alpha_n}{2})\rangle-\lambda\langle1, \dots,  1\rangle.
\end{align*}
We next solve the following system of equations:
\begin{align}
\lambda &=\frac{1}{2}\sec^2(\frac{\alpha_1}{2}), \label{eq:n-1} \\
&\vdots \\
\lambda &= \frac{1}{2}\sec^2(\frac{\alpha_n}{2}),\label{eq:n-n} \\
\alpha_1+\dots + \alpha_n &=\pi(n-2).\label{eq:n-sum}
\end{align}
From the $n$ Equations~\ref{eq:n-1} through \ref{eq:n-n}, we see
$$\sec^2(\frac{\alpha_1}{2})=\sec^2(\frac{\alpha_2}{2})=\dots=\sec^2(\frac{\alpha_n}{2}),$$
or
$$\cos^2(\frac{\alpha_1}{2})=\cos^2(\frac{\alpha_2}{2})=\dots=\cos^2(\frac{\alpha_n}{2}).$$
This implies that 
\begin{equation} \cos(\frac{\alpha_1}{2})=\pm\cos(\frac{\alpha_2}{2})=\dots=\pm\cos(\frac{\alpha_n}{2}).
\label{eq:cases}
\end{equation}
Recall that we assumed that $\frac{\pi}{2}\leq \alpha_i<\pi$. This means that if $\cos\alpha_i = \cos\alpha_j$ then $\alpha_i=\alpha_j$. In addition, we have $-\cos(\theta)=\cos(\pi-\theta)$. In order to solve Equations~\ref{eq:n-sum} and \ref{eq:cases}, we see there are a number of cases to consider. Without loss of generality, we reduce these cases to just two.

\begin{enumerate}
    \item[Case~1.] Assume that $\cos(\frac{\alpha_1}{2})=\dots=\cos(\frac{\alpha_n}{2})$, then $\alpha_1=\dots = \alpha_n$. Combining this with Equation~\ref{eq:n-sum} gives
      \begin{align*}
    \pi(n-2) &=\alpha_1+\alpha_2+\dots + \alpha_n=n\alpha_1,   \\
    \pi(\frac{n-2}{n}) &=\alpha_1=\alpha_2=\dots = \alpha_n.
    \end{align*}
    \item[Case~2.] Without loss of generality, assume that  $-\cos\Ai = \cos\Aii = \dots =\cos(\frac{\alpha_n}{2}) $. Then rewriting, 
    $$\cos(\pi-\frac{\alpha_1}{2}) = \cos\Aii =  \dots =\cos(\frac{\alpha_n}{2}).$$ 
    We then see that $2\pi-\alpha_1=\alpha_2=\dots = \alpha_n$, and so $2\pi=\alpha_1+\alpha_2$. However, we assumed that $0<\alpha_1,\alpha_2<\pi$, which means $\alpha_1+\alpha_2<2\pi$ and so we have the desired contradiction.
\end{enumerate}

In summary, we have found that $ \pi(\frac{n-2}{n}) = \pi(1 - \frac{2}{n}) =\alpha_1=\alpha_2=\dots = \alpha_n$ is the only extreme value of our function $f$ subject to the constraint $g$. To check whether this solution yields a minimum or maximum value for $f$, we compare the values of $f$ at a nearby point. We first see that 
$$f( \pi(1 - \frac{2}{n}), \dots,  \pi(1 - \frac{2}{n})) = n\tan(\frac{\pi}{2} - \frac{\pi}{n}) =n\cot(\frac{\pi}{n}).$$
Second, we note that as $\alpha_i\rightarrow \pi$, then $\tan(\frac{\alpha_i}{2})\rightarrow\infty$. Third, we set $\alpha_1=\pi-\epsilon$ for $\epsilon <\frac{2}{n}$ and set $\alpha_2=\dots=\alpha_n=\frac{1}{n-1}[(n-2)\pi-(\pi-\epsilon)]$. Then as $\epsilon\rightarrow 0$, we see that $f(\alpha_1,\dots,\alpha_n)$ can be made as large as we like. Finally, this means that 
that the function $f$ is minimized when  $\pi(1 - \frac{2}{n}) =\alpha_1=\alpha_2=\dots = \alpha_n$, and in this case, $\um$ is an equiangular $n$-gon.   

When we substitute this result into Equation~\ref{eq:n-inequality} and rearrange, we find the maximum ribbon width:
$$   0\le n\cot(\frac{\pi}{n}) \le\frac{1}{w} \quad \text{hence} \quad w\leq\frac{1}{n\cot(\frac{\pi}{n})}.$$
\end{proof}

\end{document}